\documentclass[11pt,a4paper]{article}

\usepackage{amsthm}
\usepackage{amsmath}
\usepackage{amssymb}
\usepackage{graphicx}
\usepackage{enumerate}
\usepackage{mathrsfs}
\usepackage{tikz}
\usepackage{geometry}
\usepackage{changepage}
\usepackage[numbers]{natbib}
\usepackage{hyperref}
\hypersetup{
pdfpagemode=UseThumbs,
pdftoolbar=true,        
pdfmenubar=true,        
pdffitwindow=false,     
pdfstartview={Fit},    
pdftitle={Acoustic waveguide with a dissipative inclusion},    
pdfauthor={},     
pdfsubject={},  
pdfcreator={},   
pdfproducer={}, 
pdfkeywords={}, 
pdfnewwindow=true,      
colorlinks=true,       
linkcolor=magenta,          
citecolor=red,        
filecolor=cyan,      
urlcolor=blue           
}

\allowdisplaybreaks
\theoremstyle{plain}
\newtheorem{theorem}{Theorem}[section]
\newtheorem{proposition}[theorem]{Proposition}
\newtheorem{lemma}[theorem]{Lemma}
\newtheorem{corollary}[theorem]{Corollary}
\newtheorem{example}[theorem]{Example}
\newtheorem{remark}[theorem]{Remark}
\newtheorem{definition}[theorem]{Definition}
\newtheorem{conjecture}{Conjecture}
\def\qed{\hfill{$\blacksquare$}\medskip }

\def\BET{\begin{theorem}}
\def\ENT{\end{theorem}}
\def\BEP{\begin{proposition}}
\def\ENP{\end{proposition}}
\def\BEL{\begin{lemma}}
\def\ENL{\end{lemma}}
\def\BEC{\begin{corollary}}
\def\ENC{\end{corollary}}
\def\BEE{\begin{example} \rm}
\def\ENE{\end{example}}
\def\BER{\begin{remark} \rm}
\def\ENR{\end{remark}}
\def\BED{\begin{definition} \rm}
\def\END{\end{definition}}
\def\BECJ{\begin{conjecture}}
\def\ENCJ{\end{conjecture}}

\def\bea{\begin{eqnarray}}
\def\eea{\end{eqnarray}}

\def\beas{\begin{eqnarray*}}
\def\eeas{\end{eqnarray*}}

\def\beq{\begin{equation}}
\def\eeq{\end{equation}}

\def\bbI{{\mathbb I}}
\def\bbC{{\mathbb C}}
\def\bbK{{\mathbb K}}
\def\bbR{{\mathbb R}}

\def\cA{{\mathcal A}}
\def\cB{{\mathcal B}}

\def\cR{{\mathcal R}}

\def\cT{{\mathcal T}}

\def\cV{{\mathcal V}}

\newcommand{\eps}{\varepsilon}
\newcommand{\dsp}{\displaystyle}
\newcommand{\mrm}[1]{\mathrm{#1}}

\def\tr{{\mathsf tr}}
\newcommand{\om}{\omega}
\newcommand{\Om}{\Omega}
\newcommand{\N}{\mathbb{N}}
\newcommand{\R}{\mathbb{R}}
\newcommand{\rcoef}{\mathcal{R}}
\newcommand{\tcoef}{\mathcal{T}} 
\newcommand{\rN}{r}
\newcommand{\rD}{R}

\begin{document}

~\vspace{0.3cm}
\begin{center}
{\sc \bf\LARGE 
Acoustic waveguide with a dissipative inclusion}
\end{center}

\begin{center}
\textsc{Lucas Chesnel}$^1$, \textsc{J\'er\'emy Heleine}$^2$, \textsc{Sergei A. Nazarov}$^3$, \textsc{Jari Taskinen}$^4$\\[16pt]
\begin{minipage}{0.9\textwidth}
{\small
$^1$ Inria, Ensta Paris, Institut Polytechnique de Paris, 828 Boulevard des Mar\'echaux, 91762 Palaiseau, France;\\
$^2$ Institut de Math\'ematiques de Toulouse CNRS UMR 5219, Universit\'e Paul Sabatier, F-31062 Toulouse Cedex 9, France;\\
$^3$ Institute of Problems of Mechanical Engineering RAS, V.O., Bolshoi pr., 61, St. Petersburg, 199178, Russia;\\
$^4$ Department of Mathematics and Statistics, University of Helsinki,
P.O.Box 68, FI-00014 Helsinki, Finland.\\[10pt] 
E-mails:  \texttt{lucas.chesnel@inria.fr}, \texttt{jeremy.heleine@math.univ-toulouse.fr},\hfill\\ 
\texttt{srgnazarov@yahoo.co.uk}, \texttt{jari.taskinen@helsinki.fi}\\[-14pt]
\begin{center}
(\today)
\end{center}
}
\end{minipage}
\end{center}
\vspace{0.4cm}

\begin{minipage}{0.9\textwidth}
\noindent\textbf{Abstract.} We consider the propagation of acoustic waves in a waveguide containing a penetrable dissipative inclusion. We prove that as soon as the dissipation, characterized by some coefficient $\eta$, is non zero, the scattering solutions are uniquely defined. Additionally, we give an asymptotic expansion of the corresponding scattering matrix when $\eta\to0^+$ (small dissipation) and when $\eta\to+\infty$ (large dissipation). Surprisingly, at the limit $\eta\to+\infty$, we show that no energy is absorbed by the inclusion. This is due to the so-called skin-effect phenomenon and can be explained by the fact that the field no longer penetrates into the highly dissipative inclusion. These results guarantee that in monomode regime, the amplitude of the reflection coefficient has a global minimum with respect to $\eta$. The situation where this minimum is zero, that is when the device acts as a perfect absorber, is particularly interesting for certain applications. However it does not happen in general. In this work, we show how to perturb the geometry of the waveguide to create 2D perfect absorbers in monomode regime. Asymptotic expansions are justified by error estimates and theoretical results are supported by numerical illustrations. \\

\noindent\textbf{Key words.} Acoustic waveguide, dissipation, perfect absorber, asymptotic analysis, scattering matrix, boundary layer phenomenon.
\end{minipage}


\section{Introduction}
We are interested in the propagation of acoustic waves in a waveguide of $\R^d$ unbounded in one direction and which contains a bounded penetrable inclusion made of a dissipative material as represented in Figure \ref{Domain}. Generally speaking, an incident wave propagating in such a structure gives rise to a scattered field so that a part of the incoming energy is scattered at infinity while the other part is dissipated in the inclusion (see the identity (\ref{14})). The original motivation for this study comes from the design of perfect absorbers also known as Coherent Perfect Absorbers (CPA) or time-reversed lasers \cite{CGCS10,Long10,WCGNSC11}. These are particular configurations where all the energy of the incident field is dissipated in the inclusion. Numerical and experimental strategies to construct such perfect absorbers have been proposed in the physical community, for example in \cite{CGCS10,MTRRP15,RTRMTP16,JRPG17,YCFS17} (see also references therein). \\
\newline
In this article, we wish to consider the mathematical properties of the corresponding problem. The first goal of this work is to prove that the scattering solutions are well-defined for the problem (\ref{1}) below with the dissipation characterized by the parameter $\eta$ and to prove results concerning the corresponding scattering matrix $\mathbb{S}^\eta$. More precisely we will establish that $\mathbb{S}^\eta$ is symmetric for all $\eta\in[0;+\infty)$. However $\mathbb{S}^\eta$ is unitary only for $\eta=0$, that is when there is no dissipation. This is directly related to the fact that for $\eta>0$, some energy is dissipated in the inclusion. In order to get a perfect absorber, a natural idea might be to work with a large value of $\eta$, that is with what one might think be a very dissipative inclusion. However we will see that this is a wrong intuition because when $\eta$ tends to $+\infty$, the dissipative inclusion behaves like a sound soft obstacle. As a consequence, $\mathbb{S}^\eta$ converges to a unitary matrix and not to the zero matrix (what we would like to obtain a perfect absorber).\\
\newline
Below we compute an asymptotic expansion of the scattering solution $u^{\eta}$ when $\eta$ tends to zero (this is a rather direct result) and when $\eta$ tends to $+\infty$ (this is the first main result of the paper). From this, we derive the desired expansions of $\mathbb{S}^\eta$. Let us mention that when $\eta$ goes to $+\infty$, the field tends to zero inside the inclusion and becomes localized at its boundary. In other words, a boundary layer phenomenon appears, this is the so-called skin-effect. In order to capture it, it is necessary to work with an adapted structure for the asymptotic expansion. Such problems have already been studied in the literature (see \cite{ViLu,na20,na229}) and here we adapt existing techniques  to our case. Very close to our topic, let us mention the article \cite{HJN05} (see also the related works \cite{HJN08,DHJ10,HaLe10,BIJ20}). There the authors consider a problem similar to ours but set in a bounded domain with an impedance boundary condition which can be seen as a low order approximation of the Sommerfeld radiation condition in freespace. They propose and justify asymptotic models for highly dissipative inclusions. Their justifications mainly rely on proofs by contradiction. We will proceed a bit differently by giving a direct demonstration of the essential stability estimate of Theorem \ref{thT3}. Moreover, we will focus our attention on a waveguide problem and describe in detail consequences for the scattering matrix. \\
\newline
The second main outcome of this article is the justification of a method to create perfect absorbers in monomode regime. In this regime, $\mathbb{S}^\eta$ is simply a complex number (reflection coefficient). As we will see, for a given geometry, studying the behaviour of $\eta\mapsto\mathbb{S}^\eta$ is not sufficient because in general $\eta\mapsto\mathbb{S}^\eta$ does not go through zero. What we will show is how, for any fixed parameter $\eta>0$, to modify the shape of the waveguide to get a reflection coefficient equal to zero.\\
\newline
The outline is as follows. In Section \ref{sec1}, we start by presenting the setting. Then we prove results concerning the well-posedness of the scattering problem and establish some properties of the corresponding scattering matrix. Section \ref{SectionAsymptoSmallDissip} is dedicated to the derivation of an expansion of $\mathbb{S}^\eta$ as $\eta$ tends to zero (small dissipation). Then in Section \ref{sec5}, we give the formal procedure to compute an expansion of $\mathbb{S}^\eta$ when $\eta$ tends to $+\infty$. In Section \ref{sec6}, we prove an error estimate to justify the latter expansion. Possible generalizations and open questions are discussed in Section \ref{sec8} while numerical illustrations of the results are presented in Section \ref{SectionNumerics}. In Section \ref{SectionCPA}, we propose a strategy to construct 2D perfect absorbers in monomode regime by modifying the shape of the waveguide. Finally, in Section \ref{sec9}, we establish an important stability estimate which is the crucial ingredient of the justification of the asymptotic expansion of $\mathbb{S}^\eta$ for large values of $\eta$. This is the most technical part of the work. The main results of this article are Theorem \ref{thT1} (asymptotic of $\mathbb{S}^\eta$ for small $\eta$), Theorem \ref{thT2} (asymptotic of $\mathbb{S}^\eta$ for large $\eta$) and the strategies of Section \ref{SectionCPA} to design perfect absorbers.

\section{Problem setting} \label{sec1}

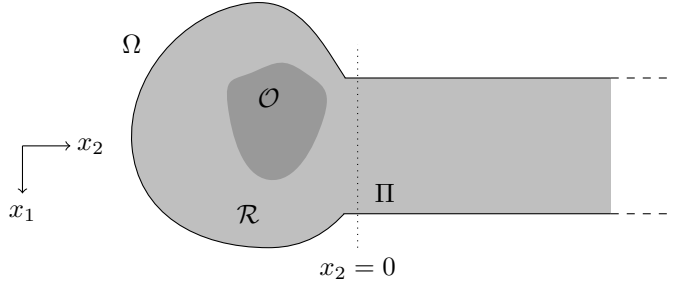
\begin{figure}[!ht]
\centering
\begin{tikzpicture}[scale=0.9]
\draw [fill=gray!50,draw=black] plot [smooth cycle, tension=1] coordinates {(-3,1) (-1,-0.5) (0.8,1) (0.5,1.5) (-1.3,3.1)};
\draw[fill=gray!50,draw=none](0,0) rectangle (4,2);
\draw [fill=gray!80,draw=none] plot [smooth cycle, tension=1] coordinates {(-1.6,1.6) (-1,0.5) (-0.2,1.4) (-0.5,2.1) (-1.2,2.1)};
\draw (0.09,0)--(4,0);
\draw (0.11,2)--(4,2);
\draw[dashed](4,0)--(5,0);
\draw[dashed](4,2)--(5,2);
\draw[dotted](0.3,-0.5)--(0.3,2.5);
\node at (0.7,0.3){\small$\Pi$};
\node at (-1.3,0.0){\small$\mathcal{R}$};
\node at (-3,2.5){\small$\Om$};
\node at (0.3,-0.8){\small$x_2=0$};
\node at (-1,1.7){\small$\mathcal{O}$};
\begin{scope}[shift={(-4.6,1)}]
\draw[->] (0,0)--(0.7,0);
\draw[->] (0,0)--(0,-0.7);
\node at (1,0){\small$x_2$};
\node at (0,-1){\small$x_1$};
\end{scope}
\node at (6,0.65){\small \phantom{$x_1$}};
\end{tikzpicture}
\caption{Example of waveguide $\Om$ in $\R^2$.\label{Domain}} 
\end{figure}

Let $\Omega := \mathscr{R} \cup \Pi $ be an acoustic waveguide of $\bbR^d$, $d\ge2$, made of the bounded resonator $\mathscr{R}\subset \bbR_-^d :=\{x=(x_1,\dots,x_d)\in\bbR^d\,|\,x_d<0\}$ and of the half-cylinder $\Pi := \om \times [0;+\infty)$ with bounded cross section $\om \subset \bbR^{d-1}$ (see Figure \ref{Domain}). We assume that $\mathscr{R}$, $\Pi$ are such that $\Omega$ is connected and that its boundary $\partial\Omega$ is Lipschitz. To model the dissipative inclusion, introduce $\mathcal{O}$ a non empty open set with smooth boundary such that $\overline{\mathcal{O}}\subset\mathscr{R}$ and $b \in \mathscr{C}^\infty( \overline{\mathcal{O}})$ a positive function that we extend by zero in $\Omega\backslash\overline{\mathcal{O}}$. We consider the Neumann boundary value problem  
\begin{equation}\label{1}
\begin{array}{|rlcl}
\Delta_x u^\eta+\lambda \big( 1 + i \eta\,b \big) u^\eta  &=& 0 & \mbox{ in }\Omega\\[3pt]
\partial_nu^\eta &=& 0  & \mbox{ on }\partial\Omega,
\end{array}
\end{equation}
where $\Delta_x$ is the Laplace operator, $\partial_n$ is the outward normal derivative and $\eta \geq 0$ is a scalar parameter. In what follows, we will study the dependence of the acoustic field $u^\eta$ with respect to the dissipation parameter $\eta$. In order to describe the propagation of waves in $\Omega$, we need to define the modes of the waveguide. To provide their expressions, first we introduce the eigenvalue problem for the Neumann Laplacian in the cross section $\om$, namely
\begin{equation}\label{EigenValPbCrossSection}
\begin{array}{|rcll}
-\Delta_y\varphi&=&\lambda\varphi&\mbox{ in }\om\\[3pt]
\partial_\nu\varphi&=&0&\mbox{ on }\partial\om.
\end{array}
\end{equation}
Here and in what follows, we use the notation $y:=(x_1,\dots,x_{d-1})$, $z:=x_d$. Additionally in (\ref{EigenValPbCrossSection}), $\partial_\nu$ stands for the outward normal derivative on $\partial\om$. We denote by $\lambda_j$ the corresponding eigenvalues and choose $\varphi_j$ associated eigenfunctions such that
\begin{equation}\label{eigenpairTransverse}
\begin{array}{l}
0=\lambda_0<\lambda_1 \le \lambda_2 \le  \cdots \le \lambda_j \le \cdots \rightarrow +\infty,\\[2pt]
(\varphi_j,\varphi_{k})_{\om}=\delta_{j,k},\qquad j,k\in\N:=\{0,1,\dots\}. 
\end{array}
\end{equation}
Above, $\delta_{j,k}$ stands for the Kronecker symbol and $(\cdot,\cdot )_\om$ is the natural scalar product of the Lebesgue space $L^2(\om)$. Let us fix
\[
\lambda\in(\lambda_{J-1};\lambda_J)
\]
for some $J\in\N^{\ast}:=\{1,2,\dots\}$. Then for $j=0,\dots,J-1$, we define the propagating waves 
\begin{equation}\label{defModes}
w^{\pm}_j(x)= (2|\alpha_j|)^{-1/2} e^{\pm i \alpha_j z}\varphi_j(y)\quad \mbox{ with }\ \alpha_j:=\sqrt{\lambda-\lambda_j}.
\end{equation}
The normalization factors in (\ref{defModes}) are introduced so that the scattering matrix defined after (\ref{9}) is unitary for $\eta=0$. The first goal of the article is to show the existence, for any $\eta\ge0$, of solutions to the problem \eqref{1} of the form
\begin{equation}\label{9}
u_j^\eta  = \chi\, w_j^- + \chi\,\sum_{k=0}^{J-1} s_{jk}^\eta w_k^+  + 
\tilde{u}_j^\eta, 
\end{equation}
where $\mathbb{S}^\eta = (s_{jk}^\eta)_{0\le j,k \le J-1}\in\mathbb{C}^{J\times J}$ is the scattering matrix and the remainder $\tilde{u}_j^\eta$ decay exponentially at infinity. In (\ref{9}), $\chi$ is a cut-off function depending only on the $z$ variable such that $\chi(x)=1$ for $z\ge1$ and $\chi(x)=0$ for $z\le0$. Note that the introduction of such a $\chi$ is necessary because a priori the waves (\ref{defModes}) are not defined in the whole $\Om$. Indeed nothing guarantee that the $\varphi_j$ extend outside of $\om$ to a function such that (\ref{defModes}) still solves the homogeneous Helmholtz equation.\\
\newline
The second goal of the article is to investigate the dependence of the scattering matrix $\mathbb{S}^\eta$ with respect to $\eta$. Classical results guarantee that for $\eta=0$, $\mathbb{S}^\eta$  is unitary and symmetric. We will derive the remarkable result that when $\eta$ tends to infinity, that is when one may think that the dissipation becomes large, $\mathbb{S}^\eta$ has a limit $\mathbb{S}^\infty$ which is also unitary and symmetric. As we will see, this result is the consequence of the appearance of a skin-effect phenomenon characterized by the concentration of the field 
$u_j^\eta$ inside the inclusion to a neighbourhood of the boundary $\partial\mathcal{O}$.

\section{Solvability and definition of the scattering matrix} \label{sec4}

First, we show the existence of solutions to the problem \eqref{1} of the form (\ref{9}). We could work with a Dirichlet-to-Neumann operator to bound the domain $\Om$ to recover compactness properties and apply Fredholm theory (see e.g. \cite{HaPG98}). Instead we choose here to work with Kondratiev, or weighted Sobolev, spaces. For $\beta\in\R$, define $W_\beta^1(\Omega)$ as the completion of $\mathscr{C}^\infty_0(\overline{\Om}):=\{\varphi|_{\Om}\,,\,\varphi\in\mathscr{C}^\infty_0(\R^d)\}$ for the norm
\bea
\Vert u ; W_\beta^1 (\Omega) \Vert = \Vert e^{\beta z} u ; H^1(\Omega)\Vert. \label{160}
\eea
Here $\mathscr{C}^\infty_0(\R^d)$ denotes the set of infinitely differentiable functions of $\R^d$ which are compactly supported, $\beta$ is the weight index and $H^1(\Omega)$ stands for the
standard Sobolev space. Clearly, we have $W_0^1(\Omega) = H^1(\Omega)$ but for $\beta > 0$,
the elements of $W_\beta^1 (\Omega)$ decay exponentially at infinity. 
Let us mention that we use this norm notation with a semicolon instead of writing it as a subscript because below we will have to deal with different parameters and it will be easier to read this way. Let us fix the weight index such that
\begin{equation}\label{HypoWeight}
\beta \in (0;\sqrt{\lambda_J-\lambda}) . 
\end{equation}
In this situation, all the evanescent modes of problem (\ref{1}) belong to $W_\beta^1 (\Omega)$. Now following \cite[Ch.\,5]{NaPl}, \cite{na489}, introduce the weighted space with detached 
asymptotics  $W^{\mrm{out}}(\Omega)$ consisting of functions 
\bea
u = \chi\sum_{j=0}^{J-1} a_j w_j^+  + \tilde{u},   \label{17}
\eea
with $a := (a_0, \ldots, a_{J-1}) \in \bbC^J$ and $\tilde{u} \in W_\beta^1(\Omega)$. We endow it with the norm  
\bea
\Vert u ; W^{\mrm{out}}(\Omega) \Vert = \big(  |a|^2 + \Vert \tilde{u} ; W_\beta^1 (\Omega)\Vert^2
\big)^{1/2} .  \label{18} 
\eea
The space $W^{\mrm{out}}(\Omega)$ allows us to impose the classical radiation condition for Problem (\ref{1}). Note  that we have $ W_{\beta}^1 (\Omega)\subset W^{\mrm{out}}(\Omega)\subset W_{-\beta}^1 (\Omega)$.  For $u\in W^{\mrm{out}}(\Omega)$, the map $\phi\mapsto \mathfrak{a}(u,\phi)$ with
\[
\mathfrak{a}(u,\phi)=\int_{\Om}\nabla u\cdot\overline{\nabla\phi}-\lambda(1+i\eta\,b)u\overline{\phi}\,dx
\] 
is well-defined in $W_{\beta}^1 (\Omega)$. Although $u\notin W_{\beta}^1 (\Omega) $ in general when $u\in W^{\mrm{out}}(\Omega)$, we will extend it as a map in $W_{-\beta}^1 (\Omega)$. For $\phi\in\mathscr{C}^\infty_0(\overline{\Om})$, applying Green's formula yields
\begin{equation}\label{Decompo}
\mathfrak{a}(u,\phi)=-\int_{\Om}(\Delta+\lambda(1+i\eta\,b))\bigg(\chi\sum_{j=0}^{J-1} a_j w_j^+\bigg)\overline{\phi}\,dx+\mathfrak{a}(\tilde{u},\phi).
\end{equation}
Note that above there is no boundary terms because $\chi$ is zero in the resonator and the $w_j^+$ satisfy the homogeneous Neumann boundary condition on $\partial\Om\cap\partial\Pi$. Since the integrand in the first two integrals is compactly supported, we infer that there is a constant $C>0$ independent of $\phi\in\mathscr{C}^\infty_0(\overline{\Om})$ such that
\begin{equation}\label{EstimCont}
|\mathfrak{a}(u,\phi)|\le C\,\Vert u ; W^{\mrm{out}}(\Omega) \Vert \Vert \phi ; W_{-\beta}^1 (\Omega)\Vert.
\end{equation}
By density of $\mathscr{C}^\infty_0(\overline{\Om})$ in $W_{-\beta}^1 (\Omega)$, we deduce that $\phi\mapsto\mathfrak{a}(u,\phi)$ can be uniquely extended as a continuous map in $W_{-\beta}^1 (\Omega)$. This discussion allows us to define the linear operator 
\begin{equation}\label{19}
\begin{array}{rccc}
\cA_\beta^\eta: & W^{\mrm{out}}(\Omega) & \to  & W_{-\beta}^1 (\Omega)^\ast \\[4pt]
 & u & \mapsto & \cA_\beta^\eta u
\end{array}
\end{equation}
where $\cA_\beta^\eta u$ is the unique element of $W_{-\beta}^1 (\Omega)^\ast$ such that 
\begin{equation}\label{DefOpWeak}
\langle \cA_\beta^\eta u,\phi\rangle_{\Om}=\mathfrak{a}(u,\phi),\qquad\forall \phi\in\mathscr{C}^\infty_0(\overline{\Om}).
\end{equation}
Here $W_{-\beta}^1 (\Omega)^\ast$ stands for the space of continuous antilinear forms over $W_{-\beta}^1 (\Omega)$ and $\langle\cdot,\cdot\rangle_\Omega$ denotes the (sesquilinear) duality pairing $W_{-\beta}^1 (\Omega)^\ast\times W_{-\beta}^1 (\Omega)$. Observe that from (\ref{Decompo}), for $v\in W_{-\beta}^1 (\Omega)$, we have
\begin{equation}\label{DefCA}
\langle \cA_\beta^\eta u,v\rangle_{\Om}=-\int_{\Om}(\Delta+\lambda(1+i\eta\,b))\bigg(\chi\sum_{j=0}^{J-1} a_j w_j^+\bigg)\overline{v}\,dx+\mathfrak{a}(\tilde{u},v).
\end{equation}
From estimate (\ref{EstimCont}), we see that $\cA_\beta^\eta : W^{\mrm{out}}(\Omega) \to W_{-\beta}^1 (\Omega)^\ast$ is continuous. Additionally, we have the following properties. 

\begin{proposition} \label{propP2}
Assume that $\beta$ satisfies (\ref{HypoWeight}).\\[3pt]
\noindent 1) The operator $\cA_\beta^0 $ is Fredholm of index zero. If its kernel is non empty, then $\ker\cA_\beta^0=\mrm{span}(v_1^\tr, \dots, v_T^\tr)$ where $v_1^\tr, \dots, v_T^\tr\in W_{\beta}^1 (\Omega)$ are the so-called trapped modes of the problem. In this situation, $F$ belongs to the range of $\cA_\beta^0 $ if and only if there holds $F(v_\tau^\tr)=0$, $\tau=1,\dots,T$.\\[3pt]
\noindent 2)  The operator $\cA_\beta^\eta$ is an isomorphism for all $\eta > 0$.
\end{proposition}
\begin{proof}
1) The first statement is a classical result, for the proof see  e.g. \cite[Ch.\,5]{NaPl} and \cite{na489,na546}. In particular, the fact that trapped modes, if they exist, decay exponentially at infinity is a consequence of the conservation of energy. \\
2) Observing that $\cA_\beta^\eta - \cA_\beta^0:W^{\mrm{out}}(\Omega)  \to  W_{-\beta}^1 (\Omega)^\ast$ is a compact operator because $\mathcal{O}$ is bounded, we deduce that $\cA_\beta^\eta$ is also Fredholm of index zero. Now if $u$ is an element of $\ker\cA_\beta^\eta$, first by taking $\phi\in\mathscr{C}^\infty_0(\Om)$ in (\ref{DefOpWeak}), we obtain $\Delta_xu + \lambda (1+i\eta\,b)u=0$ in $\Om$. Using Green's formula in the truncated waveguide $\Om_L := \{ x \in \Omega \, : \, z < L \}$, with $L>0$, then we can write
\[
\begin{array}{lcl}
0&=& \dsp\big( ( \Delta_x + \lambda (1 + i\eta\,b))u, u\big)_{\Om_L} - \big( u, ( \Delta_x + \lambda (1+i\eta\,b))u \big)_{\Om_L} \\[8pt]
 &=& \dsp 2 i\lambda\, \eta \int_{\mathcal{O}} b |u|^2 dx+\int_{\om}  \big( \overline{u(x) } \partial_z u(x) - u(x) \overline{\partial_z u(x) } \big)\Big|_{z= L} dy.
\end{array}
\]
Using the decomposition (\ref{17}) of $u$ as well as the normalisation of the modes (\ref{defModes}), by passing to the limit $L \to + \infty$, we get 
\beas
2 i\lambda\, \eta \int_{\mathcal{O}} b |u|^2 dx + i |a|^2 = 0 . 
\eeas
Thus there holds $u=0$ in $\mathcal{O}$ and the theorem of unique continuation (see \cite{Bers}, \cite[\S8.3]{CoKr13}, \cite{ARRV09} and the references therein) guarantees that $u\equiv0$ in $\Omega$. This shows the  second assertion. 
\end{proof}
\begin{proposition}
For $\eta\ge0$, for $j=0,\dots,J-1$, solutions  $u^\eta_j$ of the form (\ref{9}) exist. Moreover, the scattering matrix $\mathbb{S}^\eta$ appearing after (\ref{9}) is uniquely defined. 
\end{proposition}
\begin{proof}
For $j=0,\dots,J-1$ and $\eta\ge0$, define $F_j$ the element of $W^1_{-\beta}(\Om)^\ast$ such that
\begin{equation}\label{TermSource}
F_j(v)=\int_{\Om} (\Delta (\chi\,w_j^-)+\lambda(1+i\eta\,b)(\chi\,w^-_j))\,\overline{v}\,dx,\qquad\forall v\in W^1_{-\beta}(\Om).
\end{equation}
Proposition \ref{propP2} above ensures that for all $\eta\ge0$, there is a solution $v^\eta_j\in W^{\mrm{out}}(\Om)$ to the problem $\mathcal{A}^\eta_\beta v^\eta_j=F_j$. This is clear when $\eta>0$ or when $\eta=0$ with  $\mathcal{A}^\eta_\beta$ being injective. In those situations, by setting $u^\eta_j:=v^\eta_j+\chi\,w^-_j$, this guarantees the existence of the functions appearing in (\ref{9}) and also shows that the scattering matrix $\mathbb{S}^\eta$ is uniquely defined. For $\eta=0$, in case of existence of a non empty kernel, to verify that the compatibility conditions $F(v^{\mrm{tr}}_\tau)=0$, $\tau=1,\dots, T$, are indeed satisfied, one integrates by parts in (\ref{TermSource}) and exploits the fact that trapped modes decay exponentially at infinity (they belong to $W^1_\beta(\Om)$). Then $u^\eta_j$ is defined up to some linear combination of the trapped modes. However the latter do not modify the value of the $s^\eta_{jk}$ in (\ref{9}) forming $\mathbb{S}^\eta$ because they are exponentially decaying at infinity. 
\end{proof}
\noindent To set ideas, in case of existence of a non empty kernel for $\eta=0$, we impose the additional conditions
\begin{equation}\label{AdditionalConditions}
\int_{\mathcal{O}} b\,u_j^0\,\overline{v_\tau^\mrm{tr}}\,dx=0,\qquad \tau=1,\dots,T,
\end{equation}
so that $u^0_j$ becomes uniquely defined. Observe that it is always possible to impose relations (\ref{AdditionalConditions}) because the fact that the $v_\tau^\mrm{tr}$ are linearly independent and the theorem of unique continuation guarantee that the Gram matrix $((bv^{\mrm{tr}}_j,v^{\mrm{tr}}_\tau)_{\mathcal{O}})_{1\le j,\tau \le T}$ is invertible (recall that $b$ is positive in $\overline{\mathcal{O}}$). Note that conditions (\ref{AdditionalConditions}) are one choice among others. We consider that specific one because it will simplify the asymptotic analysis below (see the comment after (\ref{27})). \\
\newline
Now we establish a simple identity for the scattering matrix $\mathbb{S}^\eta$ which is a generalization of the conservation of energy. To proceed, introduce the symplectic (sesquilinear and anti-Hermitian) form $q$ such that
\[
q(v,v') = \int_{\om}  \big( \overline{v'(x)} \partial_z v(x) - v(x) \overline{\partial_z v'(x) } \big)\Big|_{z= L} dy
\]
where $L>0$ is given. Note that if $v$, $v'$ are two functions which solve Problem (\ref{1}), Green's formula ensures that the quantity $q(v,v')$ is independent of $L>0$ (because the dissipative inclusion does not meet $\Pi$). Using the definition (\ref{defModes}) of the modes, a direct calculus based on the orthogonality of the $\varphi_j$ shows that we have the identities
\begin{equation}\label{12}
q( w_j^\pm, w_k^\pm ) = \pm i \delta_{j,k},   \qquad 
q( w_j^\pm, w_k^\mp ) = 0, \qquad  j,k = 0, \ldots, J-1 .
\end{equation}
\BEP  \label{propP1}
For all $\eta\ge0$, the scattering matrix $\mathbb{S}^\eta$ is symmetric. Moreover there holds 
\bea\label{14}
\mathbb{S}^\eta\overline{\mathbb{S}^\eta}^\top + 2\lambda \eta \mathbb{B}^\eta= \bbI . 
\eea
Here $\bbI$ is the identity matrix of $\mathbb{C}^{J\times J}$ while $\mathbb{B}^\eta=(\mathbb{B}^\eta_{jk})_{0\le j,k\le J-1}$ denotes the positive definite Hermitian matrix such that
\bea \label{15} 
\mathbb{B}_{jk}^\eta  = \int_{\mathcal{O}} b \,u_j^\eta\,\overline{u_k^\eta}\,dx.
\eea
\ENP
\begin{remark}
Identity (\ref{14}) considered on the diagonal of the matrices can be interpreted as follows: the energy brought to the system by some incident wave is converted in some energy scattered at infinity and some energy dissipated in the inclusion.
\end{remark}
\begin{remark}
Identity (\ref{14}) shows that $\mathbb{S}^\eta$ is unitary for $\eta=0$ but not for $\eta>0$. 
\end{remark}
\begin{proof}Green's formula in the truncated waveguide $\Om_L := \{ x \in \Omega \, : \, z < L \}$, with $L>0$, gives 
\[
0=\big( ( \Delta_x + \lambda (1 + i\eta\,b))u_j^\eta, \overline{u_k^\eta}\big)_{\Om_L} 
- \big( u_j^\eta, ( \Delta_x + \lambda (1 - i\eta\,b))\overline{u_k^\eta} \big)_{\Om_L} 
= q(u_j^\eta, \overline{u_k^\eta}).
\]
Passing to the limit as $L \to + \infty$ in this identity and using relations \eqref{9}, \eqref{12}, we get $s^\eta_{jk}=s^\eta_{kj}$ which guarantees that $\mathbb{S}^\eta$ is symmetric. On the other hand, performing the limit passage $L\to+\infty$ in 
\[
\begin{array}{rcl}
0&=&\big( ( \Delta_x + \lambda (1 + i\eta\,b))u_j^\eta, u_k^\eta\big)_{\Om_L} 
- \big( u_j^\eta, ( \Delta_x + \lambda (1 + i\eta\,b))u_k^\eta \big)_{\Om_L}  \\[8pt]
 & = & \dsp 2 i \lambda \eta \int_{\mathcal{O}} b\,u_j^\eta \overline{u_k^\eta}\,dx + q(u_j^\eta, u_k^\eta),
 \end{array}
\]
we obtain
\[
0= 2 i \lambda \eta \int_{\mathcal{O}} b\,u_j^\eta \overline{u_k^\eta}\,dx -i \delta_{j,k} + i \sum_{p=0}^{J-1} s_{jp}^\eta \overline{s_{kp}^\eta}.
\]
This proves (\ref{14}). To see that $\mathbb{B}$ has the mentioned properties, we use that $b$ is positive in $\overline{\mathcal{O}}$ and observe that the restrictions $u_0^\eta|_\mathcal{O}, \dots, u_{J-1}^\eta|_\mathcal{O}$
are linearly independent due to the theorem of unique continuation (the fact that $u_0^\eta, \dots, u_{J-1}^\eta$ are linearly independent functions in $\Om$ can be established by taking the projections against the $\varphi_j$ at $z=L$ for different $L>0$). This is enough to guarantee that the Gram matrix $\mathbb{B}^\eta$ generated by the scalar product of the $L^2$-space on $\mathcal{O}$ with weight $b$ is positive definite and Hermitian. 
\end{proof}

\section{Asymptotic of the scattering matrix for small $\eta$}\label{SectionAsymptoSmallDissip}

Since $\mathcal{A}^\eta_\beta$ is a small perturbation of $\mathcal{A}^0_\beta$ when $\eta\to0^+$, the justification of the following asymptotic formula for $\mathbb{S}^\eta$ is straightforward.

\BET \label{thT1}
As $\eta$ tends to zero, we have the expansion 
\bea
\mathbb{S}^\eta = \mathbb{S}^0 -  \lambda\eta\mathbb{B}^0\mathbb{S}^0  + \tilde{\mathbb{S}}^\eta  . \label{24}
\eea
Here $\mathbb{B}^0$ is the positive definite Hermitian matrix introduced in 
\eqref{15} with $\eta=0$ and the remainder $\tilde{\mathbb{S}}^\eta$ is such that $\Vert \tilde{\mathbb{S}}^\eta ; \bbC^{J \times J} \Vert \leq c \eta^2$. 
\ENT
\begin{proof}
For the function $u^\eta_j$ introduced in (\ref{9}), we consider the ansatz
\[
u_j^\eta = u_j^0  + \eta u_j'+ \dots,
\]
where dots  stand for higher-order asymptotic terms. Inserting this expansion into \eqref{1} and extracting the terms of order $\eta$, we obtain the problem 
\begin{equation}\label{27}
\begin{array}{|rlcl}
-\Delta_x u_j'-\lambda u_j' &=& i\lambda bu_j^0 & \mbox{ in }\Omega\\[3pt]
\partial_nu_j' &=& 0  & \mbox{ on }\partial\Omega.
\end{array}
\end{equation}
In case of existence of trapped modes, the compatibility conditions of Proposition \ref{propP2} item $1)$ are fulfilled due to (\ref{AdditionalConditions}). This guarantees that Problem (\ref{27}) always admits a solution $u_j'\in W^{\mrm{out}}(\Om)$ with 
\[
u_j' = \chi\sum_{p=0}^{J-1} s_{jp}' w_p^+ + \tilde{u}_j', \qquad \tilde{u}_j' \in W_\beta^1(\Omega).
\]
To get the expression of the coefficients $s_{jp}'$, we can write
\[
i\lambda( b u_j^0, u_k^0)_\mathcal{O} = - \lim_{L \to + \infty} \big( (\Delta_x + \lambda) u_j', u_k^0 \big)_{\Om_L} = 
- \lim_{L \to + \infty} q(u_j',u_k^0) = - i \sum_{p=0}^{J-1} s_{jp}'\overline{s_{kp}^0 } . 
\]
Therefore if we set $\mathbb{S}'=(s'_{jp})_{0\le j,p\le J-1}$, we obtain $- \mathbb{S}'\overline{\mathbb{S}^0}^\top  = \lambda\mathbb{B}^0$. Since $\mathbb{S}^0$ is unitary, this implies $\mathbb{S}'=-\lambda\mathbb{B}^0\mathbb{S}^0$. 
\end{proof}

\section{Asymptotic of the scattering matrix for large $\eta$}\label{sec5}

In the previous section, we studied the behaviour of $\mathbb{S}^\eta$ as $\eta\to0$. Now we are interested in the situation where $\eta\to+\infty$. In this case, we expect that the large index material (see (\ref{1})) makes the field tend to zero inside the inclusion $\mathcal{O}$. This leads us to consider the problem with mixed boundary conditions
\begin{equation}\label{29}
\begin{array}{|rlcl}
\Delta_x u^\infty+\lambda u^\infty  &=& 0 & \mbox{ in }\Omega_\bullet:=\Om\setminus\overline{\mathcal{O}}\\[3pt]
\partial_nu^\infty &=& 0  & \mbox{ on }\partial\Omega\\[3pt]
u^\infty &=& 0  & \mbox{ on }\partial\mathcal{O}.
\end{array}
\end{equation}
As for (\ref{1}) with $\eta=0$ (see Proposition \ref{propP2}  item $1)$), this problem admits the diffraction solutions 
\bea\label{30}
u_j^\infty  =  \chi\,w_j^- + \chi\sum_{k=0}^{J-1} s_{jk}^\infty\,w_k^+  + \tilde{u}_j^\infty,
\eea
where $s_{jk}^\infty\in\mathbb{C}$, $\tilde{u}_j^\infty \in W_\beta^1(\Omega)$. The scattering matrix $\mathbb{S}^\infty:=(s_{jk}^\infty)_{0\le j,k\le J-1}\in\bbC^{J \times J}$ is symmetric and unitary. Note that for Problem (\ref{29}) there may exist trapped modes, that we denote by $m_1^\tr, \dots, m_K^\tr \in W_\beta^1(\Omega_\bullet) \subset H^1(\Omega_\bullet)$ (the family is assumed to be linearly independent). In this case, we impose additionally the conditions 
\begin{equation}\label{conditionsTrappedModes}
\int_{\partial\mathcal{O}} b^{-1/2}\partial_nu^\infty_j\,\overline{\partial_n m^{\mrm{tr}}_k}\,ds=0,\qquad k=1,\dots,K,
\end{equation}
so that $u_j^\infty$ becomes uniquely defined. Remark that we can indeed have (\ref{conditionsTrappedModes}) because the Gram matrix $((b^{-1/2}\partial_n m^{\mrm{tr}}_j, \partial_n m^{\mrm{tr}}_k)_{\partial\mathcal{O}} )_{1\le j,k\le K}$ is invertible since $\partial_n m^{\mrm{tr}}_1|_{\partial\mathcal{O}},\dots,\partial_n m^{\mrm{tr}}_K|_{\partial\mathcal{O}} $ are linearly independent functions. To show this property, use again the theorem of unique continuation and the fact that the $m^{\mrm{tr}}_\tau$ are linearly independent in $\Om_\bullet$. Again conditions (\ref{conditionsTrappedModes}) are one choice among others but we impose that specific one because it simplifies the asymptotic procedure below (see (\ref{CondiExistence})).\\
\newline
To construct the asymptotics of the scattering solutions $u_0^\eta, \dots, u_{J-1}^\eta$ and of $\mathbb{S}^\eta$ as $\eta \to + \infty$, outside the inclusion $\mathcal{O}$, we consider the ansatz
\bea\label{31}
u_j^\eta = u_j^\infty + \eta^{-1/2} u_j' + \dots \qquad\mbox{ in }\Omega_\bullet.  
\eea
Let us mention that we use the same notation as in the previous section for the correctors ($u'_j$ and $\mathbb{S}'$ below) but the objects are different. Since we expect the first term of the ansatz of $u_j^\eta$ inside $\mathcal{O}$ to be zero and since $\partial_nu^{\infty}_j$ cannot vanish on the entire surface $\partial\mathcal{O}$ due to the second boundary condition in (\ref{29}) together with the theorem of unique continuation, we need to compensate for the jump of the normal derivative on $\partial\mathcal{O}$. To proceed, we employ the method of Vishik-Lyusternik \cite{ViLu} and construct a boundary layer in the vicinity of $\partial \mathcal{O}$. Introduce a system of curvilinear coordinates in a neighborhood $\cV$ of $\partial \mathcal{O}$ (recall that $\mathcal{O}$ is smooth) such that $n$ is the oriented distance to $\partial \mathcal{O}$, $n > 0$
in  $\mathcal{O} \cap \cV$, and $s$ is some local coordinate on $\partial\mathcal{O}$. Observe that with this definition, on $\partial\mathcal{O}$, $\partial_n$ stands for the normal derivative with a normal pointing to the interior of $\mathcal{O}$. The Laplace operator in the local coordinates writes 
\bea\label{34}
\Delta_x = \partial_n^2 + L(n,s,\partial_n, \partial_s), 
\eea
where $L(n,s,\partial_n, \partial_s) = \ell_n(n,s)\partial_n + \ell_s(n,s) 
\partial_s+  \ell_{sn}(n,s)\partial_n\partial_s+ \ell_{ss} (n,s)\partial_s^2$. Here $\ell_n$, $\ell_s$, $\ell_{sn}$, $\ell_{ss}$  are smooth functions of $n\in [0;H]$, $H>0$ and $s \in \partial \mathcal{O}$. Introducing the stretched normal coordinate 
\[
t := \eta^{1/2} n,
\]
in view of (\ref{34}), we have 
\bea\label{36}
\Delta_x + \lambda(1 + i\eta\,b(x) )= \eta(\partial_t^2 + i \lambda b(s) ) + \ldots \ . 
\eea
Here and in what follows, to simplify we write $b(s)$ for the value of $b(x)$ for $x\in\partial\mathcal{O}$ ($n=0$). Inside $\mathcal{O}$, following the procedure presented in \cite{ViLu}, we consider the expansion
\bea\label{37}
u_j^\eta(x) = \eta^{-1/2} E(t,s) \partial_n  u_j^\infty|_{\partial\mathcal{O}}(x) + \dots
\eea
where $E$ is the solution of the ordinary differential equation (compare with (\ref{36}))
\bea\label{39}
\partial_t^2 E(t,s) + i \lambda b(s) E(t,s) = 0, \qquad t \in (0;+\infty),
\eea
which decays as $t \to + \infty$ and such that $\partial_tE(t,s)|_{t=0}=1$. A direct calculus yields  
\bea\label{38}
E(t,s) = \frac{\sqrt{2}}{(-1+i)\sqrt{\lambda\,b(s)}}\,\exp\Big( \cfrac{-1+i}{\sqrt{2}}\,\sqrt{\lambda\,b(s)} \,t \Big)= -\frac{1+i}{\sqrt{2\lambda\,b(s)}}\,\exp\Big( \cfrac{-1+i}{\sqrt{2}}\,\sqrt{\lambda\,b(s)} \,t \Big) .
\eea
Note that the main terms in (\ref{31}) and (\ref{37}) have the same normal derivative on the surface $\partial\mathcal{O}$. Now we are going to set the correction term in the expansion \eqref{31} to make the two-term 
asymptotic approximation continuous on $\partial \mathcal{O}$. We see that $u_j'$ must satisfy the problem  
\begin{equation}\label{40}
\begin{array}{|rlcl}
\Delta_x u_j'+\lambda u_j'  &=& 0 & \mbox{ in }\Omega_\bullet\\[3pt]
\partial_nu_j' &=& 0  & \mbox{ on }\partial\Omega\\[3pt]
u_j' &=& g'_j(x)  & :=E(0,s)\,\partial_n
u_j^\infty|_{\partial\mathcal{O}}(x)\qquad \mbox{ for }x\in\partial\mathcal{O}.
\end{array}
\end{equation}
According to \cite[Ch.\,5]{NaPl} or Proposition \ref{propP2} item $1)$, in case of absence of trapped modes for (\ref{29}), Problem (\ref{40}) admits a unique solution in $W^{\mrm{out}}(\Omega_\bullet)$ (this space is defined similarly to $W^{\mrm{out}}(\Om)$, see (\ref{17})). In case of existence of linearly independent trapped modes $m_1^{tr}, \dots, m_K^{tr}$, a solution exists if and only if we have 
\begin{equation}\label{CondiExistence}
\int_{\partial\mathcal{O}} g_j'\,\partial_nm^{\mrm{tr}}_k\,ds=0,\qquad k=1,\dots, K
\end{equation}
(note that we can impose that trapped modes be real valued). But thanks to our choice (\ref{conditionsTrappedModes}), these conditions are indeed satisfied. Therefore Problem (\ref{40}) admits the solution 
\[
u_j'  =  \chi\sum_{k=0}^{J-1} s_{jk}' w_k^+  + \tilde{u}_j,  
\]
where $\tilde{u}_j' \in W_\beta^1(\Omega_\bullet)$ and the 
coefficients  $s_{jk}'$ compose the correction term $\mathbb{S}'$ in the representation of the scattering matrix 
\bea\label{44}
\mathbb{S}^\eta = \mathbb{S}^\infty + \eta^{-1/2} \mathbb{S}' + \tilde{\mathbb{S}}^\eta. 
\eea
These coefficients are computed in the usual way by writing
\[
\int_{\partial \mathcal{O}} g_j'\, \overline{ \partial_n
u_k^\infty}\,ds = \lim_{L \to + \infty}
\int_{\om} \big( \overline{ u_k^\infty} \partial_z u_j'
- u_j' \overline{ \partial_z u_k^\infty} \big) \Big|_{z =L} dy=
\lim_{L \to + \infty} q(u_j', u_k^\infty) = 
i \sum_{p=0}^{J-1} s_{jp}'\overline{s_{kp}^\infty}.
\]
This gives 
\bea\label{45}
i \mathbb{S}'\overline{\mathbb{S}^\infty}^\top  = -\frac{1+i}{\sqrt{2 \lambda}}\,\mathbb{E} \qquad\mbox{ and so }\qquad \mathbb{S}' = \frac{-1+i}{\sqrt{2 \lambda}}\,\mathbb{E} \,\mathbb{S}^\infty 
\eea
where $\mathbb{E}$ is the positive definite Hermitian matrix of size $J \times J$ with entries
\bea \label{46}
\mathbb{E}_{jk} = \int_{\partial \mathcal{O}} b^{-1/2}\, 
\partial_n u_j^\infty \, \overline{\partial_n u_k^\infty}\,
ds.
\eea
The following assertion  will be proved in the next section. 

\BET \label{thT2}
As $\eta$ tends to $+\infty$, we have the expansion 
\begin{equation}\label{ExpansionLarge}
\mathbb{S}^\eta = \mathbb{S}^\infty + \eta^{-1/2}\,\frac{-1+i}{\sqrt{2 \lambda}}\,\mathbb{E} \,\mathbb{S}^\infty + \tilde{\mathbb{S}}^\eta.
\end{equation}
Here $\mathbb{E}$ is the positive definite Hermitian matrix introduced in (\ref{46}) and the remainder $\tilde{\mathbb{S}}^\eta$ is such that $\Vert \tilde{\mathbb{S}}^\eta ; \bbC^{J\times J } \Vert \leq 
c\eta^{-3/4} $. 
\ENT
\begin{remark}
As mentioned in the introduction, this shows that a very dissipative inclusion behaves as a sound soft obstacle and when $\eta$ tends to $+\infty$, $\mathbb{S}^\eta$ converges to the unitary scattering matrix associated to it. 
\end{remark}
\begin{remark}
In the energy identity $\mathbb{S}^\eta\overline{\mathbb{S}^\eta}^\top + 2\lambda \eta \mathbb{B}^\eta= \bbI$ (see (\ref{14})), one may think at first glance that the term $\lambda \eta \mathbb{B}^\eta$ becomes large when $\eta$ goes to $+\infty$. However this is not true because according to the definition (\ref{15}) of $\mathbb{B}^\eta$ and the result of Remark \ref{RmkDecayInside}, we have $\Vert \mathbb{B}^\eta ; \bbC^{J\times J } \Vert \leq c\eta^{-3/4} $. This is in agreement with the expansion (\ref{ExpansionLarge}).
\end{remark}

\section{Justification of asymptotics as $\eta \to + \infty$}  \label{sec6}

The most technical and complicated part in this article is the
derivation of an priori estimate for the solutions of Problem \eqref{1} for large values of $\eta$. To proceed, we will work in an \textit{ad hoc} weighted space with a composite structure as in \cite{MaNaPl} which combines the spaces of \cite{AgVi64} and \cite{Ko}. It will allow us to take into account the different behaviours of the $u_j^\eta$ on each side of $\partial\mathcal{O}$. In addition to the detached asymptotics at infinity together with the
exponential factor as in the norm \eqref{160}, we will include powers of the parameter $\eta$ and the function
\bea
\varrho(x) = \min \{ 1, {\rm dist}\,(x,\partial \mathcal{O}) \}. \label{47}
\eea
Note that the $\min$ with respect to one in this definition ensures that $\varrho$ has an influence only in a neighbourhood of $\partial\mathcal{O}$. We formulate the important a priori estimate in Theorem \ref{thT3} below but postpone its proof to the Appendix. In order to make the demonstration more comprehensible, we will make the assumption that trapped modes do not exist for Problem \eqref{29}. The general case can be
treated with marginal alterations that will be commented in item $5^\circ$ of the Appendix. \\
\newline
To state our estimate, first we introduce the norms we will work with. For $u\in W^1_\beta(\Omega)$, set 
\begin{equation}\label{510}
\Vert u; W_{\beta,\eta}^{1}(\Omega) \Vert = \big( \Vert e^{\beta z } \nabla_x u; L^2(\Omega) \Vert^2 + \Vert e^{\beta z } (\varrho + \eta^{-1/2})^{-1} u; L^2(\Omega_\bullet) \Vert^2 + \eta \Vert u ; L^2(\mathcal{O}) \Vert^2 \big)^{1/2}.
\end{equation}
Then for $u = \chi\sum_{j=0}^{J-1} a_j w_j^+  + \tilde{u}$ with $a = (a_0, \ldots, a_{J-1}) \in \bbC^J$ and $\tilde{u} \in W_\beta^1(\Omega)$, we define 
\bea\label{51}  
 \big\Vert u; W_{\eta}^{\mrm{out}}(\Omega) \Vert =
\big( |a|^2 + \Vert \tilde{u} ; W_{\beta,\eta}^{1}(\Omega)\Vert^2
\big)^{1/2}.
\eea
Observe that for all $\eta>0$, the norms $W^1_\beta(\Omega)$ and $W^1_{\beta,\eta}(\Omega)$ (respectively $W^{\mrm{out}}(\Omega)$ and $W^{\mrm{out}}_\eta(\Omega)$) are equivalent. However the constants of equivalence depend on $\eta$ and degenerate when $\eta$ tends to $+\infty$. 
Below, for $F\in W^1_{-\beta}(\Om)^\ast$, we shall also work with the norm
\[
\Vert F  ; W_{-\beta,\eta}^{1} (\Omega)^*  \Vert = \sup_{v\in W^1_{-\beta}(\Om)\setminus\{0\}}\cfrac{|\langle F,v\rangle_\Omega|}{\Vert v; W_{-\beta,\eta}^{1}(\Omega) \Vert}\,.
\]

\BET  \label{thT3}
Assume that Problem \eqref{29} does not have trapped modes. Then there is a constant $c>0$ such that for all $F\in W^1_{-\beta}(\Om)^\ast$, the function $u\in W^{\mrm{out}}(\Om)$ satisfying $\mathcal{A}^\eta_\beta u=F$ is such that
\bea\label{52}
\Vert u; W_{\eta}^{\mrm{out}} (\Omega) \Vert \leq c\,\Vert F  ; W_{-\beta,\eta}^{1} (\Omega)^*  \Vert,
\eea
for all $\eta\ge \eta_0$, where $\eta_0$ is given. 
\ENT

\noindent Now let us turn to the justification of the asymptotic expansion given in Theorem \ref{thT2}.\\
\newline
\textit{Proof of Theorem \ref{thT2}.} Denote by $u_\bullet^\eta$ (resp. $u_\circ^\eta$) the restriction of $u^\eta$ to the domain $\Omega_\bullet$ (resp. $\mathcal{O}$). In accordance with the formal procedure of Section \ref{sec5}, we define an asymptotic approximation $\hat{u}^\eta_j$ of the function $u^\eta_j$ appearing in (\ref{9}) by setting
\begin{equation}\label{53}
\hat{u}^\eta_j=\begin{array}{|ll}
\,\hat{u}^\eta_{j\bullet} & \mbox{ in }\Om_\bullet \\[4pt]
\,\hat{u}^\eta_{j\circ} & \mbox{ in }\mathcal{O}\\[4pt]
\end{array}\qquad\mbox{ with }\qquad\begin{array}{|ll}
\,\hat{u}^\eta_{j\bullet} & = u_j^\infty + \eta^{-1/2} u_j'\\[4pt]
\,\hat{u}^\eta_{j\circ}(x) & = \eta^{-1/2} \chi_\mathcal{O}(x) E(\sqrt{\eta} n,s) \partial_n  u_j^\infty|_{\partial\mathcal{O}}(s), \quad x \in \mathcal{O}.\\[4pt]
\end{array}
\end{equation}
Here the terms $u_j^\infty$, $u_j'$, $E$ are respectively introduced in (\ref{30}), (\ref{40}), (\ref{38}). Moreover $\chi_\mathcal{O} \in \mathscr{C}^\infty(\overline{\mathcal{O}})$ is a cut-off function with support in $\overline{\mathcal{O}} \cap \cV$ such that $\chi_\mathcal{O}(x)=1$ for $n \in [0;d]$ with $d >0$ (for the definition of $\cV$ and the local coordinates $(n,s)$, see before (\ref{34})). \\
\newline
Define the error $e_j^\eta:=u_j^\eta-\hat{u}_j^\eta$. The function $e_j^\eta$ is smooth on each side of $\partial\mathcal{O}$ and according to the boundary conditions in (\ref{29}), (\ref{40}), we have 
\bea\label{54} 
\hat{u}^\eta_{j\bullet} = \hat{u}^\eta_{j\circ}\qquad \mbox{ on } \partial \mathcal{O}.  
\eea
Moreover we observe that in the expansion of $e_j^\eta$ at infinity, the incoming waves of $u_j^\eta$, $\hat{u}_j^\eta$ cancel. These properties are sufficient to guarantee that $e_j^\eta$ is an element of $W^{\mrm{out}}(\Om)$. On the other hand, using that 
\[
\Delta_x \hat{u}^\eta_{j}+\lambda \hat{u}^\eta_{j}=0\qquad\mbox{ in }\Om_\bullet
\]
and defining the functions (see in particular (\ref{37}) for the calculus of the second one)
\begin{equation}\label{defDiscrepancy}
f := - \Delta_x \hat{u}^\eta_{j\circ} - \lambda(1+i\eta\,b\big) \hat{u}^\eta_{j\circ},\qquad\qquad g:=\partial_n \hat{u}_{j\bullet}^\eta|_{\partial\mathcal{O}}-\partial_n \hat{u}_{j\circ}^\eta|_{\partial\mathcal{O}} := \eta^{-1/2} 
\partial_n u_j',
\end{equation}
we get 
\[
\mathcal{A}^\eta_\beta e^\eta_j=F\qquad\mbox{where}\qquad F(\psi)=\int_{\mathcal{O}} f \overline{\psi} \,dx  +  
\int_{\partial\mathcal{O}} g\overline{\psi}\,ds,\qquad\forall\psi\in W^1_{-\beta}(\Om).
\]
Let us estimate the $W^1_{-\beta,\eta}(\Om)^\ast$-norm of $F$. First, by using formulas \eqref{34}, \eqref{39}, we can write
\[
\begin{array}{l}
 \big| \eta^{-1/2} \chi_\mathcal{O}(x) \big( \Delta_x + \lambda(1+ i\eta\,b(x) )\big) (E(\sqrt{\eta}n,s) \partial_ n  u_j^\infty |_{\partial\mathcal{O}}(s)) \big|
\leq c e^{-\delta \sqrt{\eta} n },  \\[8pt]
 \big| \eta^{-1/2} [\Delta_x, \chi_\mathcal{O}(x) ] (E(\sqrt{\eta}n,s) \partial_ n   u_j^\infty|_{\partial\mathcal{O}}(s)) \big|
\leq c e^{-\delta \sqrt{\eta} d} 
\end{array}
\]
with $\delta > 0$. Here $[\Delta_x, \chi_\mathcal{O}]$ stands for the commutator such that $[\Delta_x, \chi_\mathcal{O}]\varphi=\Delta_x(\chi_\mathcal{O}\varphi)-\chi_\mathcal{O}\Delta_x\varphi$. This allows us to write
\begin{equation}\label{inter0}
\Big|\int_{\mathcal{O}} f \overline{\psi} \,dx \Big| \leq  c\,\Big|\int_{\cV\cap\mathcal{O}} e^{- 2 \delta \sqrt{\eta} n}\,dx\Big|^{1/2} \,\Vert \psi ;L^2(\mathcal{O}) \Vert \le c \eta^{-3/4}\, \Vert \psi  ; W^1_{-\beta,\eta}(\Omega) \Vert  . 
\end{equation}
Note that above the integral over $\cV\cap\mathcal{O}$ gives a term of order $\eta^{-1/2}$. The additional factor $\eta^{-1/2}$ appears to compensate the term $\eta^{1/2}$ multiplying $\Vert \psi;L^2(\mathcal{O})\Vert$ in the norm \eqref{510}. On the other hand, from the expression of $g$ in (\ref{defDiscrepancy}), we can write
\begin{equation}\label{inter2}
\Big| \int_{\partial \mathcal{O}} g \overline{\psi}\, ds  \Big| \leq  
c \eta^{-1/2}  \Vert \psi ;L^2(\partial \mathcal{O}) \Vert.
\end{equation}
By working in local coordinates, rectifying the boundary $\partial\mathcal{O}$ and proceeding to an explicit calculus, one can establish the trace inequality (see e.g. \cite[Lem.\,5.2]{HJN05} or the proof of \cite[Thm.\,1.5.1.10]{Gris85})
\bea\label{55}
\Vert \psi ; L^2(\partial \mathcal{O})\Vert^2 \leq c_\mathcal{O} \Vert \psi; H^1(\mathcal{O}) \Vert \, \Vert \psi; L^2(\mathcal{O})\Vert.
\eea
Using it in (\ref{inter2}) and exploiting the definition of the norm (\ref{510}), we obtain 
\begin{equation}\label{inter3}
\Big| \int_{\partial \mathcal{O}} g \overline{\psi}\, ds  \Big| \le c \eta^{-3/4}\, \Vert \psi  ; W^1_{-\beta,\eta}(\Omega) \Vert  .   
\end{equation}
Gathering (\ref{inter0}) and (\ref{inter3}) leads to $\Vert F ; W_{-\beta,\eta}^{1} (\Omega)^*\Vert \leq c \eta^{-3/4}$. From the uniform estimate (\ref{52}) of Theorem \ref{thT3}, this implies 
\begin{equation}\label{ErrorEstimate}
\Vert e^\eta_j; W_{\eta}^{\mrm{out}} (\Omega)\Vert=
\Vert u_j^\eta-\hat{u}_j^\eta ; W_{\eta}^{\mrm{out}} (\Omega)\Vert \leq c \eta^{-3/4}.
\end{equation}
Finally we can take advantage of the use of the spaces with detached asymptotics: since the norm $\Vert e^\eta_j; W_{\eta}^{\mrm{out}} (\Omega) \Vert $ involves the moduli of the reflection coefficients $e^\eta_j$, this justifies the asymptotic expansion of the scattering matrix $\mathbb{S}^\eta$ given in Theorem \ref{thT2} as $\eta$ tends to $+\infty$. \hfill \qed

\begin{remark}\label{RmkDecayInside}
From (\ref{510})--(\ref{51}), we see that estimate (\ref{52}) implies in particular that the diffraction solutions $u_j^\eta$ introduced in (\ref{9}) are such that $\|u_j^\eta ; L^2(\mathcal{O})\| \le c\,\eta^{-1/2}$, where $c$ is independent of $\eta\ge\eta_0$. But this is not optimal. Indeed using the error estimate (\ref{ErrorEstimate}) together with the definition of $\hat{u}^\eta_j$ in (\ref{53}), we get $\|u_j^\eta ; L^2(\mathcal{O})\| \le c\,\eta^{-3/4}$.
\end{remark}

\section{Possible generalizations and open questions}  \label{sec8}

$1^\circ$. Dirichlet boundary condition. The problem consisting of the partial differential equation of \eqref{1} supplemented with the Dirichlet boundary condition 
\[
u^\eta = 0\qquad\mbox{ on }\partial\Omega
\]
instead of the Neumann one can be studied in the same way. All formulas remain basically the same, the only change being that the first eigenvalue $\lambda_1^D$ of the model Dirichlet problem on the cross-section $\om$
is positive so that to study waves scattering phenomena in $\Om$, one needs to assume that $\lambda  > \lambda_1^D > 0$.\\ 
\newline
$2^\circ$. Smoothness of boundaries. Concerning the study of the properties of the operator $\mathcal{A}^\eta_\beta$ in (\ref{19}), one can allow the surface $\partial \mathcal{O}$ to be only Lipschitz. Besides, the analysis of the asymptotics of the scattering matrix as $\eta \to 0^+$ can also be done as above when $\partial \mathcal{O}$ is only Lipschitz. However to consider the situation $\eta \to +\infty$, the smoothness of $\partial \mathcal{O}$ is definitely needed. To explain this observation, consider the case $d=2$ with $\mathcal{O}$ coinciding with a rectangle. Then the solution $u^\infty_j$ of Problem \eqref{29} has the
singularities $c\,r^{2/3} \sin ( 2 \theta / 3)$, where $(r,\theta) \in 
\bbR_+ \times (0;3\pi/2)$ are the polar coordinates associated with one corner of the rectangle (see \cite{Ko} and \cite[Ch.\,2]{NaPl}). Notice that the opening of the corner in $\Omega_\bullet$ is $3\pi/2$. The singularity of the normal derivative $\partial_n u_j^\infty$ is of the form 
$c\,r^{-1/3} \sin ( 2 \theta / 3)$ and therefore Problem \eqref{40} has no solution belonging to $H_{\rm loc}^1(\overline{\Omega_\bullet} )$. As a consequence the asymptotic procedure of Section \ref{sec5} fails. To compensate for these singularities one needs to construct two-dimensional power-law boundary layers in the vicinity of the corner points. Although similar boundary layers near angular and conical points of the boundaries have been studied in many papers, see \cite{butuz, 
na24, na150, na229} and others, this particular boundary  layer has not been examined yet even in the planar case, and determining its structure remains an open question. Note that the problem involved in the description of the above mentioned boundary layer near the corner points of the rectangle writes
\bea
- \Delta_\xi V - \lambda i B V = F\qquad\mbox{ in }\in\bbR^2, \label{00}
\eea
where $B$ is the function such that $B=b(P) > 0$ in the quadrant
$\bbK := \{ \xi \,|\, \xi_1 > 0,\,\xi_2 > 0 \}$ and $B=0$ in
$\bbR^2\setminus\overline{\bbK}$. Here $P$ stands for the considered corner point. 
Equation \eqref{00} has been obtained from \eqref{1} by the coordinate change $x \mapsto \xi = \eta^{1/2} (x-P)$ and the formal limit passage $\eta \to + \infty$.\\
\newline
$3^\circ$.  Impedance boundary condition. To model the fact that a part $\Gamma$ of $\partial\Om$ is dissipative, one can consider the following problem with an impedance boundary condition 
\begin{equation}\label{02}
\begin{array}{|rlll}
\Delta_x u^\eta+\lambda  u^\eta  &=& 0 & \mbox{ in }\Omega\\[3pt]
\partial_nu^\eta &=& i\eta\,bu^\eta  & \mbox{ on }\Gamma\\[3pt]
\partial_nu^\eta &=& 0  & \mbox{ on }\partial\Om\setminus\Gamma.
\end{array}
\end{equation}
Here $\Gamma $ is a bounded $(d-1)$-dimensional submanifold of $\partial \Omega$ with a smooth boundary $\partial \Gamma$ and we assume that $b> 0$ on $\overline{\Gamma}$. Then results analogous to Theorem \ref{thT1} and \ref{thT2} hold for the corresponding scattering matrix. More precisely, when $\eta \to 0^+$, similarly to \eqref{24}, we have $\mathbb{S}^\eta = \mathbb{S}^0 - \eta \mathbb{B}^0\mathbb{S}^0  + \tilde{\mathbb{S}}^\eta$ where $\mathbb{S}^0$ stands for the scattering matrix of the Neumann problem \eqref{1} with $\eta=0$ and $\mathbb{B}^0$ is the Hermitian positive matrix such that
\beas
\mathbb{B}_{jk}^0 = \int_{\Gamma} b u_j^0\,\overline{u_k^0}\, ds.
\eeas
On the other hand, when $\eta \to +\infty$, the asymptotic structure become much more involved. If $\Gamma$ is a manifold without boundary (for example, Problem \eqref{02} is posed in the domain $\Omega_\bullet$ introduced in (\ref{29}) and $\Gamma = \partial \mathcal{O}$), then we have the expansion
\bea
\mathbb{S}^\eta = \mathbb{S}^\infty + \eta^{-1} \mathbb{E}\,\mathbb{S}^\infty  + \tilde{\mathbb{S}}^\eta
\eea
where $\mathbb{S}^\infty$ is the scattering matrix of a mixed boundary value problem  similar to \eqref{29} with the condition $u^\infty=0$ on $\Gamma$ (instead of $u^\infty=0$ on $\partial\mathcal{O}$) and $\mathbb{E} = (\mathbb{E}_{jk})_{0\le j,k\le J-1}$ is the Hermitian positive matrix such that 
\bea
\mathbb{E}_{jk} = \int_{\Gamma} b^{-1} \partial_n u_j^\infty\, \overline{\partial_n u_k^\infty}\, ds.  \label{03}
\eea
However when $\overline{\Gamma}\cap\overline{\partial\Om\setminus\Gamma}\neq\emptyset$, the normal derivatives of the
solutions \eqref{30} become strongly singular on $\partial \Gamma$ so that
integrals \eqref{03} may diverge. As a consequence, one needs to add new boundary layers near $\partial \Gamma$. To clarify this, consider again a waveguide $\Omega \subset \bbR^2$ and select a connected arc $\Gamma\subset\partial \Omega$. Denote by $P^\pm$ the endpoints of $\Gamma$ where the Dirichlet and Neumann conditions meet. At these points the quantities $\nabla_x u_j^\infty$ get singularities behaving as
$| x - P^\pm|^{-1/2}$, see for example \cite[Ch.\,2]{NaPl}. Consequently the integrals \eqref{03} indeed diverge. Now assume in addition that $\partial \Omega$ is straight in a neighbourhood of $\partial\Gamma$. Then performing the coordinate change $x \mapsto \xi = \eta (x-P^\pm)$ and taking the limit $\eta \to + \infty$, one is led to consider the following problem
\begin{equation}\label{04}
\begin{array}{|l}
\,- \Delta_\xi V = F \qquad\mbox{ in }\{\xi = (\xi_1, \xi_2)\in\R^2\,|\,\xi_2>0\}\\ [5pt]
\,- \cfrac{\partial V }{\partial \xi_2} (\xi_1, 0) = 0\quad\mbox{ for }\xi_1 < 0, \qquad - \cfrac{\partial V }{\partial \xi_2} (\xi_1, 0) = i b^\pm
V(\xi_1,0) \quad\mbox{ for }\xi_1 > 0.  
\end{array}
\end{equation}
with $b^\pm := b(P^\pm)>0$. The solvability and behaviour at infinity of solutions of Problem \eqref{04} can be studied directly by the
approach of \cite{na20,na46}. Then the formal procedure for the 
construction of the asymptotics of the solutions to the original problem
with a large parameter $\eta$ can be adapted, with small modifications, 
from \cite{na24, Dauge}.\\
\newline
$4^\circ$. Monomode regime. Assume that $\lambda \in (0;\lambda_1)$ ($J=1$) so that only the modes 
\bea\label{06}
w_0^\pm (x) = ( 2 \sqrt{\lambda}\,{\rm mes}_{d-1}(\om))^{-1/2} e^{\pm i \sqrt{\lambda} z}
\eea
can propagate in the cylinder $\om \times \bbR$. In this situation, the scattering matrix reduces to a scalar reflection coefficient that we still denote by $\mathbb{S}^\eta\in\mathbb{C}$. Furthermore,  identity \eqref{14} of Proposition \ref{propP1} writes
\bea \label{07}
|\mathbb{S}^\eta|^2 +2\lambda\eta  \int_{\mathcal{O}} b\, |u^\eta|^2 
\,dx= 1.
\eea 
Here
\bea\label{08}
u^\eta = w_0^- + \mathbb{S}^\eta w_0^+   + \tilde{u}^\eta   
\eea
is the solution \eqref{9} with $j=0$. Note that the waves \eqref{06}
are defined everywhere in $\Omega$. Therefore the cut-off function $\chi$ of \eqref{9} can be omitted in \eqref{08}. Relation (\ref{07}) imposes that $|\mathbb{S}^\eta|\in[0;1)$ for $\eta>0$. Moreover Theorems \ref{thT1} and \ref{thT2} guarantee that 
\[
\lim_{\eta\to0^+}|\mathbb{S}^\eta| =\lim_{\eta\to+\infty}|\mathbb{S}^\eta| =1.
\]
Thus, the continuous function $\eta \mapsto |\mathbb{S}^\eta|$ has a global minimum at some $\eta_\star\in (0;+\infty)$. If $\mathbb{S}^{\eta_\star} = 0$, the device acts as a perfect absorber of the piston mode. However in general this minimum is not zero. We will explain in Section \ref{SectionCPA} how to modify the geometry to create perfect absorbers. 
To conclude with this discussion, let us give a remarkable formula (that we will not use in our study) for the derivative of the reflection coefficient with respect to the dissipation. The quantity $\partial_\eta u^\eta$ solves the problem
\[
\begin{array}{|rlcl}
\,- \Delta_x(\partial_\eta u^\eta) - \lambda \big( 1  + i\eta\,b \big) 
(\partial_\eta u^\eta)  &=& i \lambda b u^\eta &\mbox{ in }\partial \Omega \\[3pt]
\, \partial_n (\partial_\eta u^\eta)  &=& 0 &\mbox{ on }\partial \Omega
\end{array}
\]
and admits the representation 
\beas
\partial_\eta u^\eta = \partial_\eta \mathbb{S}^\eta\, w_0^+ + \widetilde{\partial_\eta u^\eta} .  
\eeas
Since we have $\overline{u^\eta  } = w^+_0 + \overline{\mathbb{S}^\eta}\,w^-_0 + 
\overline{\tilde{u}^\eta}$, the Green formula together with relations \eqref{12} for $j=0$ yield
\bea
-\lambda \int_{\mathcal{O}} b (u^\eta)^2 dx = - i \lim_{L \to + \infty}
q\big(\partial_\eta u^\eta, \overline{u^\eta} \big)  = \partial_\eta \mathbb{S}^\eta .
\label{05}
\eea
Notice that the integrand in \eqref{07} contains the factor $|u^\eta|^2$ while \eqref{05} involves $(u^\eta)^2$. 

\section{Numerics}\label{SectionNumerics}
In this section, we illustrate the results we have obtained above. To proceed, we compute numerically the scattering solutions $u_j^\eta$ introduced in (\ref{9}). We use a P2 finite element method in a domain obtained by truncating $\Om$. On the artificial boundary created by the truncation, a Dirichlet-to-Neumann operator with 15 terms serves as a transparent condition (see more details for example in \cite{Gold82,HaPG98,BoLe11}). Once we have computed $u_j^\eta$, we get easily the coefficients $s^\eta_{jk}$ constituting the scattering matrix $\mathbb{S}^\eta$. The numerics have been made using the library \texttt{Freefem++} \cite{Hech12}.\\
\newline
In Figure \ref{Total Field}, we display the field $u_0^\eta$ for different values of $\eta$ in a geometrical setting where only one mode can propagate ($J=1$). We also present the solution $u_0^\infty$ introduced in (\ref{30}) of the problem with a Dirichlet boundary condition on the obstacle. In accordance in particular with estimate (\ref{ErrorEstimate}), we see that $u_0^\eta$ converges to $u_0^\infty$ as $\eta$ tends to $+\infty$. Interestingly, we observe that this result seems to hold true also in the situation where $\mathcal{O}$ is a rectangle (see the right column) though this case does not enter our analysis (in our work $\mathcal{O}$ has to be smooth, see item $2^\circ$ of Section \ref{sec8}).\\
\newline
In Figure \ref{ScatteringCoeff1}, the setting is the same as in Figure \ref{Total Field}. Since we work in monomode regime, the scattering matrix $\mathbb{S}^\eta$ reduces to a complex number (reflection coefficient). Let us remark that in this simple geometry, for $\eta=0$, we have $u_0^0=w^-_0+w^+_0$ so that $\mathbb{S}^0=1$. This is observed on line 2. As expected, we  note that as $\eta$ tends to $+\infty$, $\mathbb{S}^\eta$ converges to $\mathbb{S}^\infty$ (see lines 2 and 3). For the rectangular inclusion, we 
notice that there is one value of $\eta$ such that $\mathbb{S}^\eta\approx0$ (see lines 1 and 2, column 2). This particular phenomenon does not happen in general (see column 1).\\
\newline
In Figure \ref{ScatteringCoeffMulti}, we display the eigenvalues of $\mathbb{S}^\eta$ in the complex plane for different values of $\eta$. Here the spectral parameter $\lambda$ has been chosen such that five modes can propagate ($\mathbb{S}^\eta$ is of size $5\times5$). For $\eta=0$, in this simple geometry we have $\mathbb{S}^0=\mrm{Id}$ so that the five eigenvalues are equal to one. For $\eta>0$, we observe that the modulus of the eigenvalues of $\mathbb{S}^\eta$ is smaller than one. This was expected because for $\eta>0$, $\mathbb{S}^\eta$ is no longer unitary, and more precisely this is in agreement with identity (\ref{14}). Finally, we see that when $\eta$ tends to $+\infty$, the eigenvalues of $\mathbb{S}^\eta$ converge to values on the unit circle. This corroborates the result of Theorem \ref{thT2} which guarantees that $\mathbb{S}^\eta$ tends to a unitary matrix.

\begin{figure}[!ht]
\centering
\includegraphics[width=0.48\textwidth]{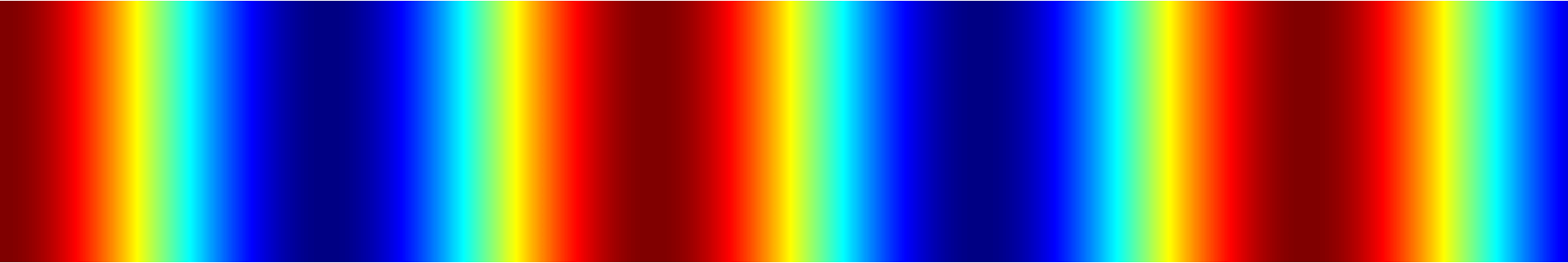}\quad\includegraphics[width=0.48\textwidth]{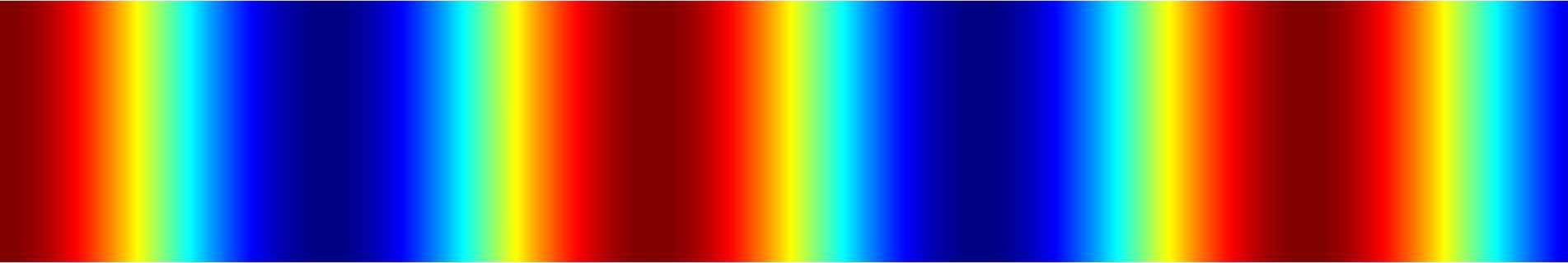}\\[2pt]
\includegraphics[width=0.48\textwidth]{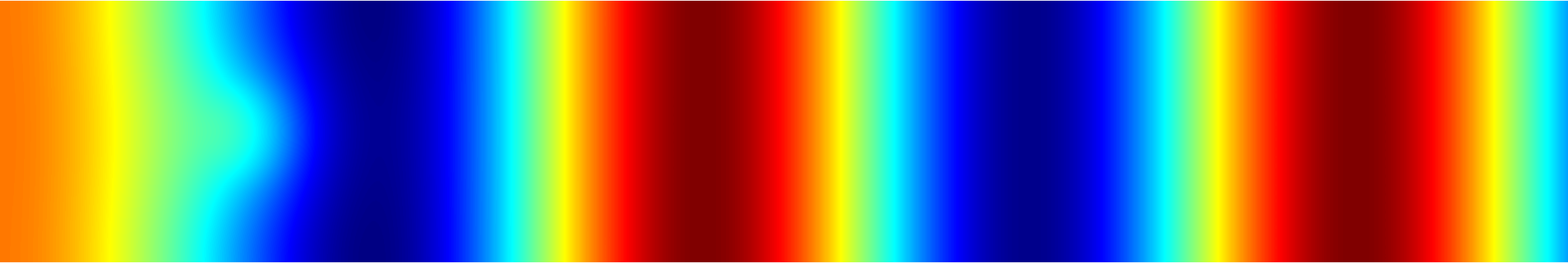}\quad\includegraphics[width=0.48\textwidth]{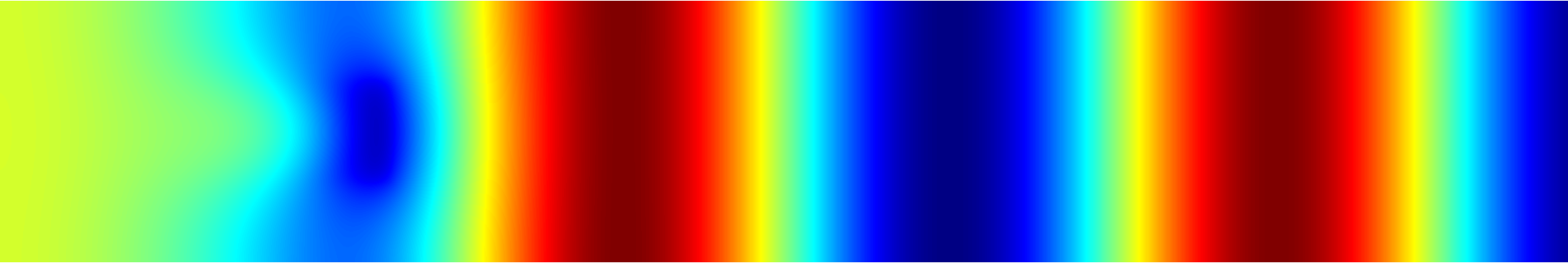}\\[2pt] 
\includegraphics[width=0.48\textwidth]{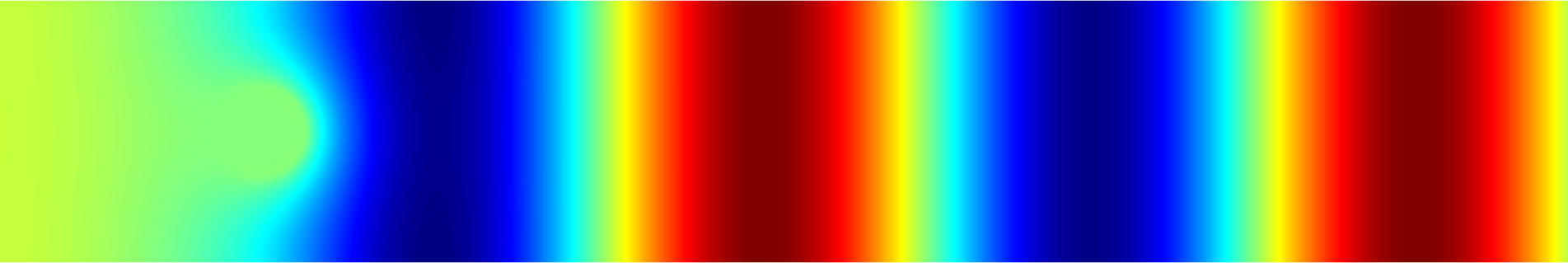}\quad \includegraphics[width=0.48\textwidth]{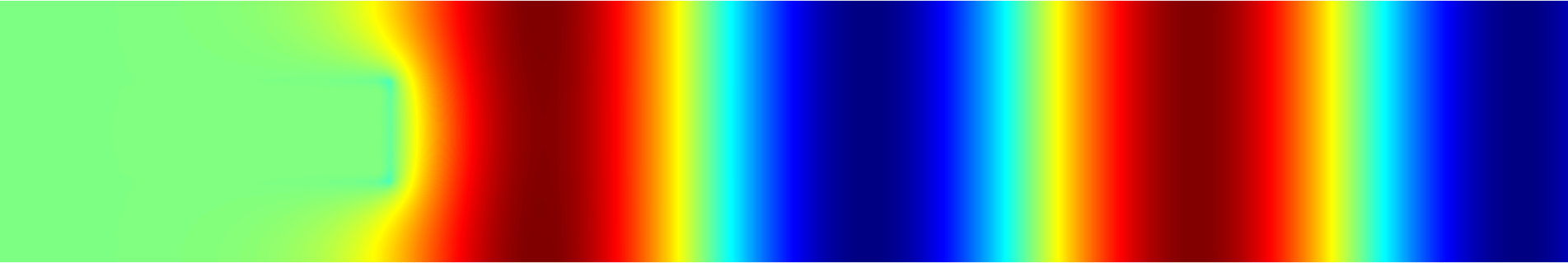}\\ [2pt]
\includegraphics[width=0.48\textwidth]{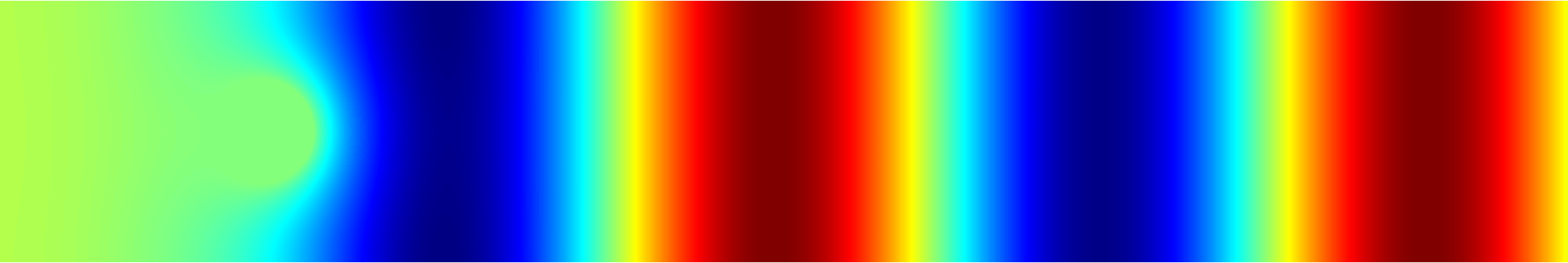}\quad \includegraphics[width=0.48\textwidth]{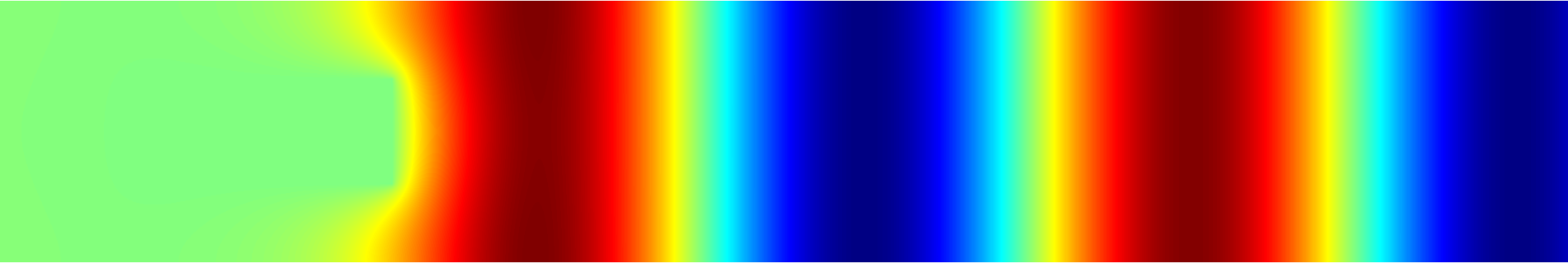}\\[2pt] 
\includegraphics[width=0.48\textwidth]{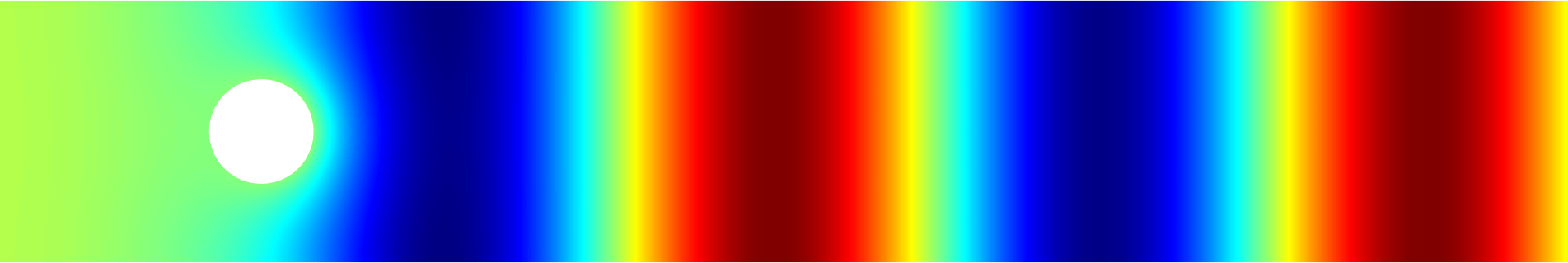}\quad \includegraphics[width=0.48\textwidth]{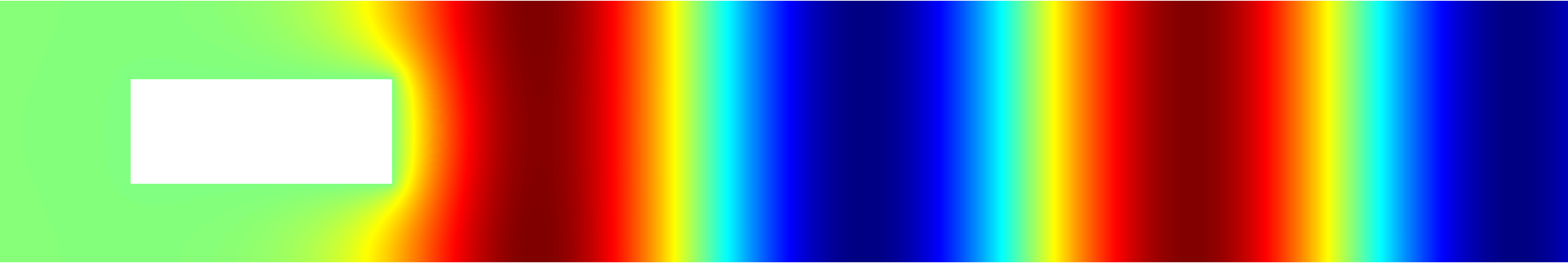}
\caption{Line 1-4: real part of $u_0^\eta$ for, respectively, $\eta=0, 10, 1000, 10^{10}$. Last line: real part of $u_0^\infty$ (solution of the problem with a Dirichlet boundary condition on the obstacle, see (\ref{30})). In the results of the left (resp. right) column, the inclusion $\mathcal{O}$ is a disk (resp. a rectangle). Here the strip is of height one ($\lambda_1=\pi$) and $\lambda=(0.8\pi)^2$ so that only one mode can propagate ($J=1$).\label{Total Field}}
\end{figure}

\begin{figure}[!ht]
\centering
\includegraphics[trim={0.38cm 0.25cm 0cm 0cm},clip,width=0.3\textwidth]{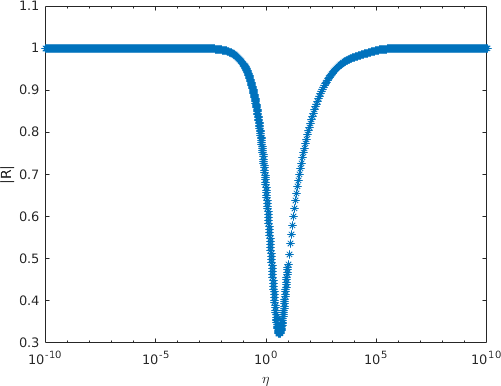}\quad\includegraphics[trim={0.38cm 0.25cm 0cm 0cm},clip,width=0.3\textwidth]{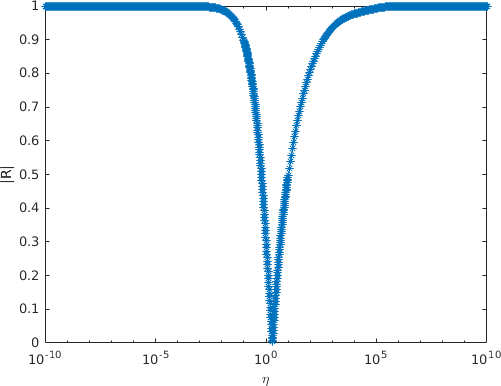}\\[2pt]
\includegraphics[width=0.3\textwidth]{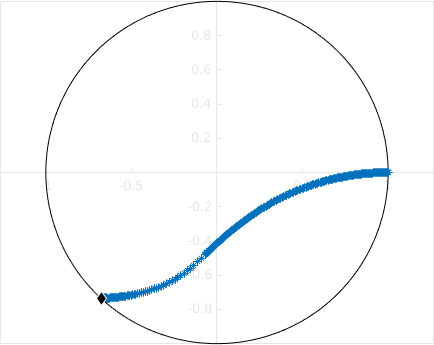}\quad\includegraphics[width=0.3\textwidth]{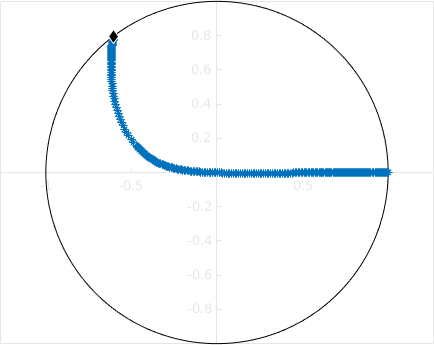}\\[5pt] 
\includegraphics[trim={0.5cm 0.25cm 0cm 0cm},clip,width=0.3\textwidth]{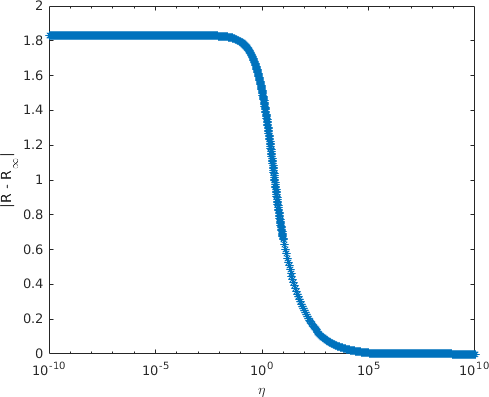}\quad \includegraphics[trim={0.5cm 0.25cm 0cm 0cm},clip,width=0.3\textwidth]{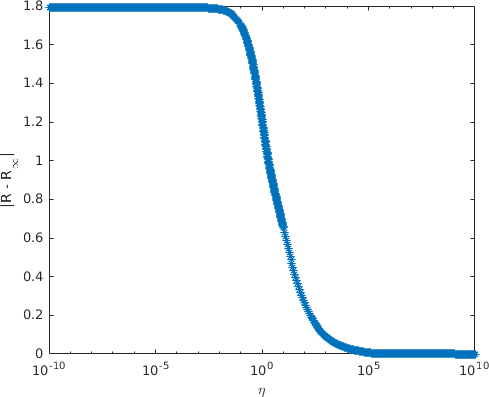} 
\caption{Same setting as in Figure \ref{Total Field} (monomode regime, circular inclusion on the left, rectangular inclusion on the right). Line 1:  $|\mathbb{S}^\eta|$ with respect to $\eta\in(10^{-10};10^{10})$. 
Line 2:  $\mathbb{S}^\eta$ in the complex plane with respect to $\eta\in(10^{-10};10^{10})$. The black diamond on the unit circle represents the value of $\mathbb{S}^\infty$.  
Line 3:  $|\mathbb{S}^\eta-\mathbb{S}^\infty|$ with respect to $\eta\in(10^{-10};10^{10})$.\label{ScatteringCoeff1}}
\end{figure}

\begin{figure}[!ht]
\centering
\includegraphics[trim={1.95cm 1.2cm 1.4cm 0.8cm},clip,width=0.24\textwidth]{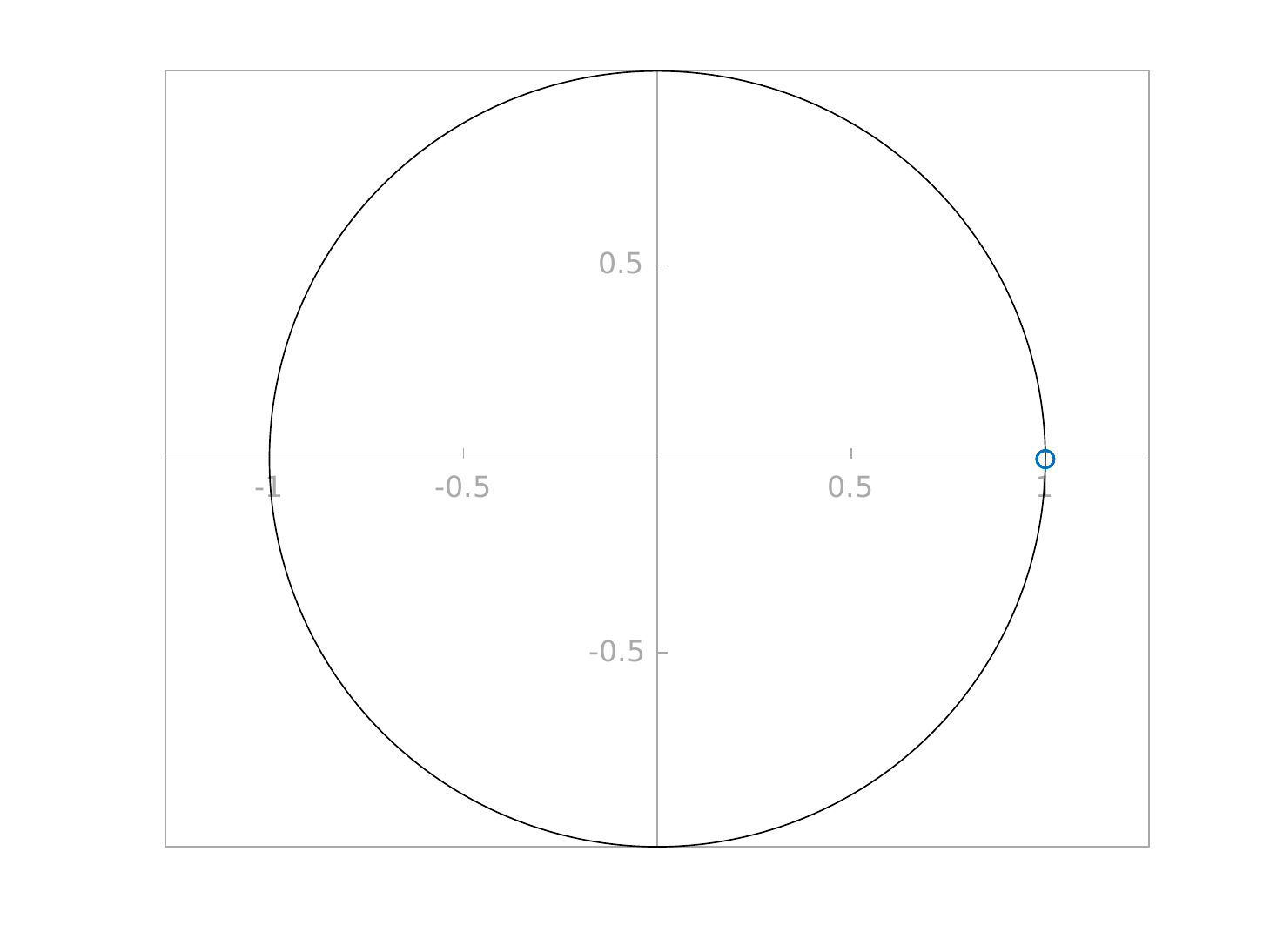}\ \includegraphics[trim={1.95cm 1.2cm 1.4cm 0.8cm},clip,width=0.24\textwidth]{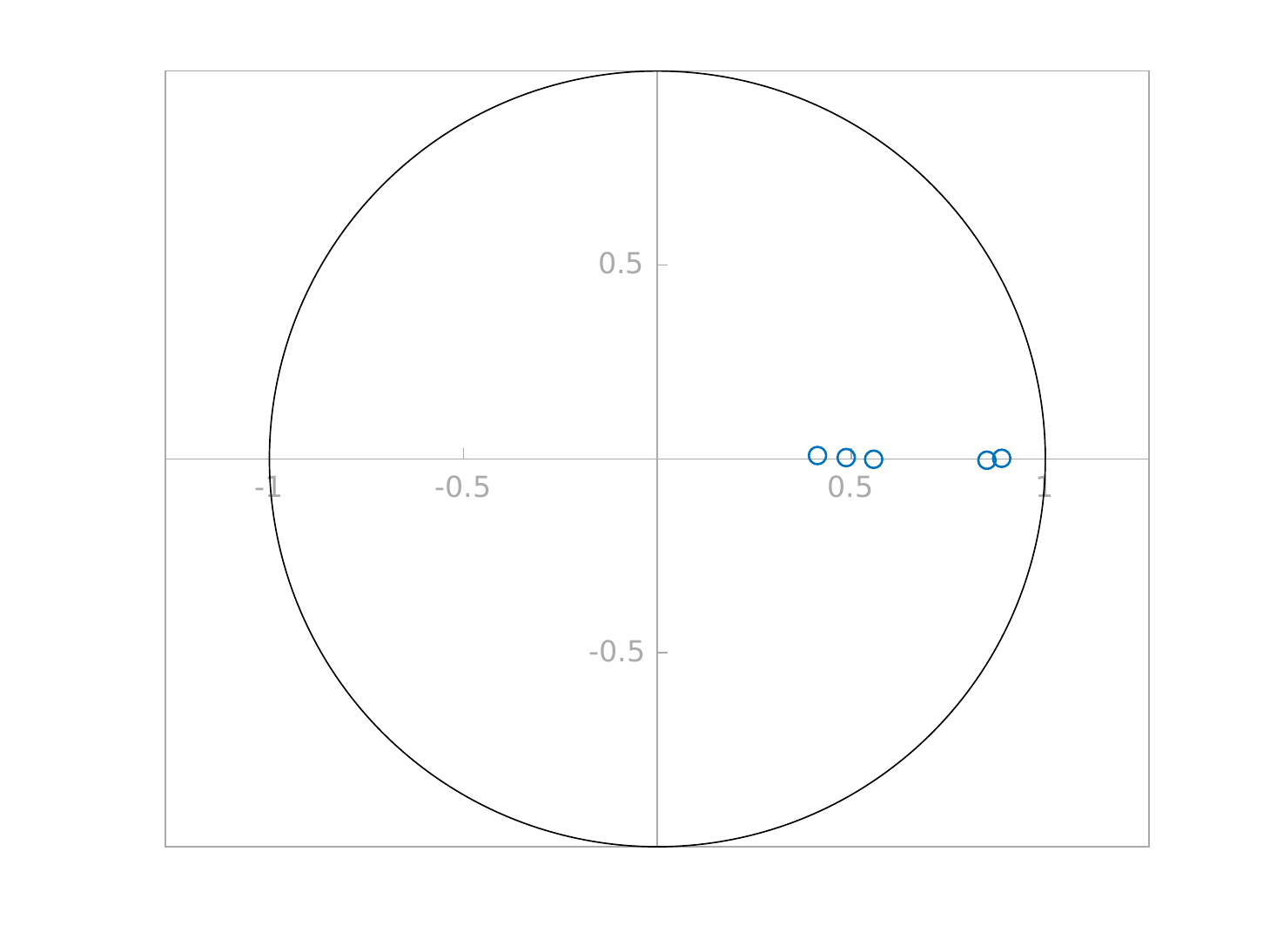}
\ \includegraphics[trim={1.95cm 1.2cm 1.4cm 0.8cm},clip,width=0.24\textwidth]{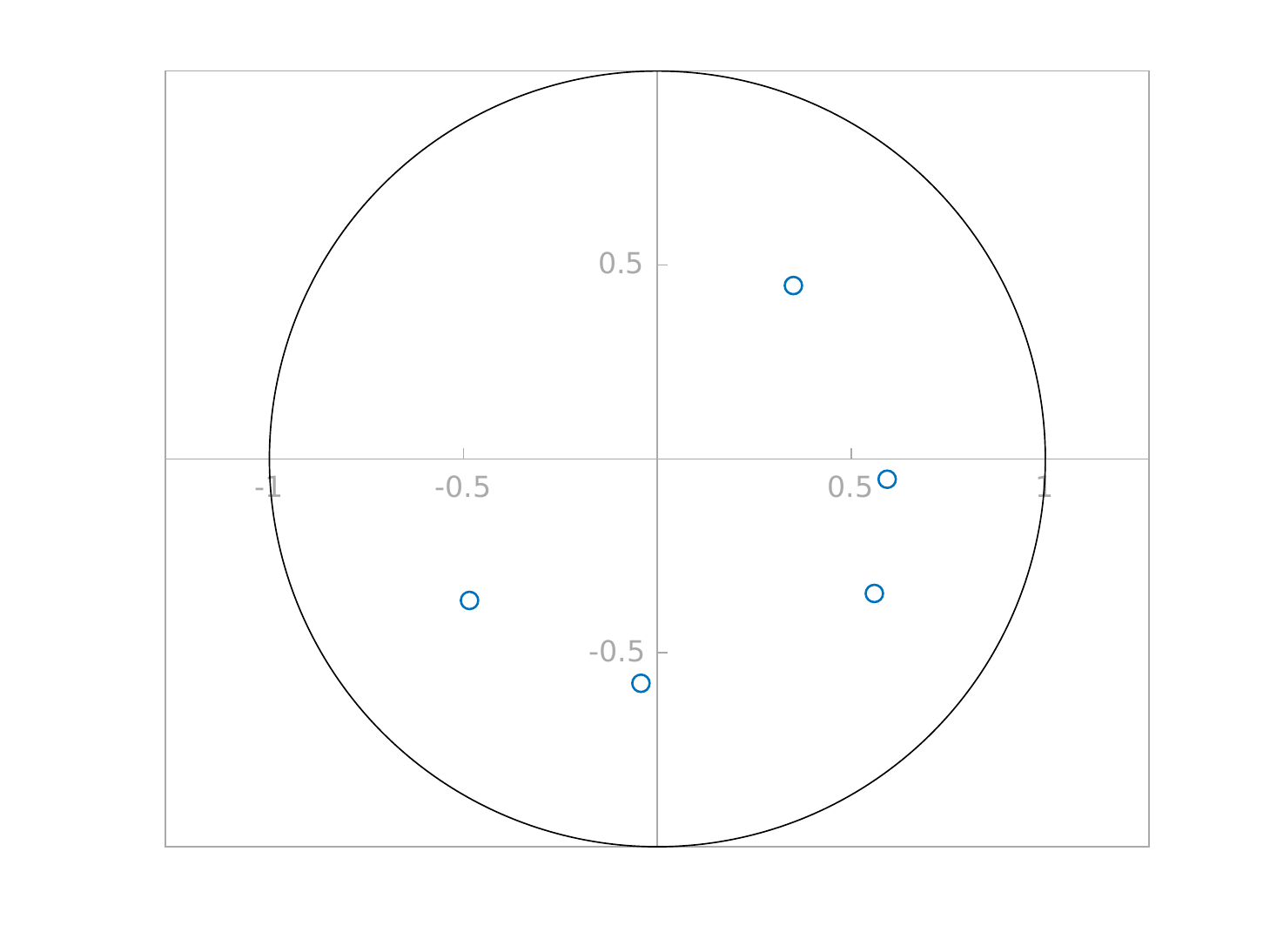}\ \includegraphics[trim={1.95cm 1.2cm 1.4cm 0.8cm},clip,width=0.24\textwidth]{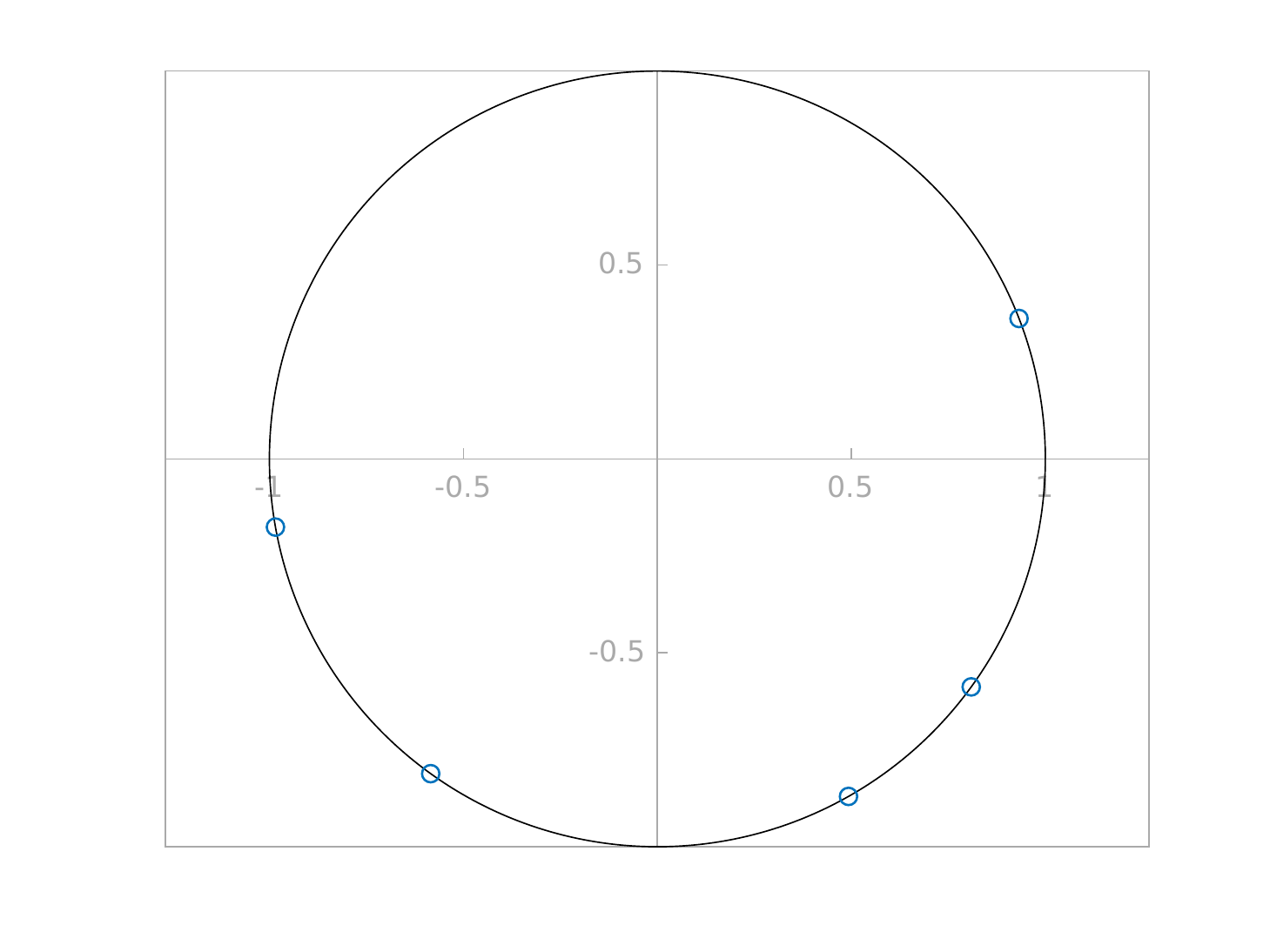}
\caption{Eigenvalues of $\mathbb{S}^\eta$ in the complex plane for $\eta=10^{-10}$, $10^{-1}$, $5$, $10^{10}$ (from left to right). Here the spectral parameter $\lambda=(4.8\pi)^2\in((4\pi)^2;(5\pi)^2)$ has been chosen such that five modes can propagate ($\mathbb{S}^\eta$ is of size $5\times5$). The black circle represents the unit circle.\label{ScatteringCoeffMulti}}
\end{figure}

\newpage

\section{Perfect absorber in 2D}\label{SectionCPA}

In this section, to simplify the presentation we work in 2D, though the ideas would work similarly in 3D, and to set ideas we assume that $\Om$ coincides with $[0;+\infty)\times(0;1)$ for $z\ge0$. In this situation, we have $\lambda_1=\pi^2$. We take $\lambda \in (0;\pi^2)$ ($J=1$) so that only the plane modes $w_0^\pm$ can propagate. As already said, in this circumstance, the scattering matrix reduces to a scalar reflection coefficient that we still denote by $\mathbb{S}^\eta\in\mathbb{C}$. Our goal here is to explain how to exhibit perfect absorbers, that is geometries where $\mathbb{S}^\eta\approx0$. We have seen in item $4^\circ$ of Section \ref{sec8} that the function $\eta\mapsto |\mathbb{S}^\eta|$ always reaches a minimum on $[0;+\infty)$. But for given $\Om$, $\mathcal{O}$, in general this minimum is not zero. What we prove below is that for any choice of $\Om$, $\mathcal{O}$, $\eta>0$, one can perturb the boundary $\partial\Om$ to get a reflection coefficient approximately equal to zero. For numerical studies, we refer the reader to \cite{CGCS10,MTRRP15,RTRMTP16,JRPG17}.\\
\newline
Once for all, we assume that $\Om$, $\mathcal{O}$ and $\eta>0$ are given.  The key idea of the approach will be to glue $\Om$ to a well-designed waveguide for which the reflection coefficient coincides with $\overline{\mathbb{S}^\eta}$. To obtain such a waveguide, we follow an idea proposed in \cite{ChPa18} that we recall now. 

\subsection{Reaching any reflection coefficient}\label{paragraphAnyR}

~\vspace{-0.6cm}

\begin{figure}[!ht]
\centering
\begin{tikzpicture}[scale=1.4]
\draw[fill=gray!30,draw=none](-2,1) rectangle (1,2);
\draw[fill=gray!30,draw=none](-4,1) rectangle (-2,2);
\draw[fill=gray!30,draw=none](-2,1.8) rectangle (-1,3.8);
\draw (-4,2)--(-2,2)--(-2,3.8)--(-1,3.8)--(-1,2)--(1,2); 
\draw (-4,1)--(1,1); 
\draw[dashed] (1,1)--(1.5,1); 
\draw[dashed] (1,2)--(1.5,2);
\draw[dashed] (-4.5,2)--(-4,2);
\draw[dashed] (-4.5,1)--(-4,1);
\draw[dotted,>-<] (-2,4)--(-1,4);
\draw[dotted,>-<] (-0.8,1.95)--(-0.8,3.85);
\node at (-1.5,4.2){\small $\ell$};
\node at (-0.4,2.9){\small $L-1$};
\draw[dashed,line width=0.5mm,gray!80] (-1.5,1)--(-1.5,3.8);
\begin{scope}[xshift=-8cm]
\draw[->] (3,1.2)--(3.6,1.2);
\draw[->] (3.1,1.1)--(3.1,1.7);
\node at (3.65,1.3){\small $z$};
\node at (3.25,1.6){\small $y$};
\end{scope}
\begin{scope}[xshift=-3.6cm,yshift=1.7cm,scale=0.8]
\draw[line width=0.2mm,->] plot[domain=0:pi/4,samples=100] (\x,{0.2*sin(20*\x r)}) node[anchor=west] {\hspace{-2.4cm}$1$};
\end{scope}
\begin{scope}[xshift=-3.6cm,yshift=1.3cm,scale=0.8]
\draw[line width=0.2mm,<-] plot[domain=0:pi/4,samples=100] (\x,{0.2*sin(20*\x r)}) node[anchor=west] {$\hspace{-2.4cm}\mathcal{R}$};
\end{scope}
\begin{scope}[xshift=0cm,yshift=1.5cm,scale=0.8]
\draw[line width=0.2mm,->] plot[domain=0:pi/4,samples=100] (\x,{0.2*sin(20*\x r)}) node[anchor=west] {$\mathcal{T}$};
\end{scope}
\end{tikzpicture}\qquad\begin{tikzpicture}[scale=1.4]
\draw[fill=gray!30,draw=none](-3,0) rectangle (0,1);
\draw[fill=gray!30,draw=none](-0.4,0) rectangle (0,2.8);
\draw (-3,1)--(-0.4,1)--(-0.4,2.8)--(0.02,2.8);
\draw (0.02,0)--(-3,0);
\draw[line width=0.6mm] (0,0)--(0,2.8);
\draw[dashed] (-3.5,1)--(-3,1); 
\draw[dashed] (-3.5,0)--(-3,0); 
\node at (0.3,1.35){\small $\Sigma_{L}$};
\draw[dotted,<->] (-0.5,3)--(0.1,3);
\draw[dotted,<->] (0.7,0)--(0.7,2.8);
\node at (0.9,1.35){\small $L$};
\node at (-0.2,3.2){\small $\ell/2$};
\phantom{\draw[dashed] (0,3.5)--(0,3); }
\begin{scope}[xshift=-2.5cm,yshift=0.7cm,scale=0.8]
\draw[line width=0.2mm,->] plot[domain=0:pi/4,samples=100] (\x,{0.2*sin(20*\x r)}) node[anchor=west] {\hspace{-2.6cm}$1$};
\end{scope}
\begin{scope}[xshift=-2.5cm,yshift=0.3cm,scale=0.8]
\draw[line width=0.2mm,<-] plot[domain=0:pi/4,samples=100] (\x,{0.2*sin(20*\x r)}) node[anchor=west] {$\hspace{-2.6cm}r/R$};
\end{scope}
\end{tikzpicture}
\caption{Geometries of $\mathcal{G}_L$ (left) and $\mathfrak{g}_{L}$ (right). \label{DomainOriginal}} 
\end{figure}
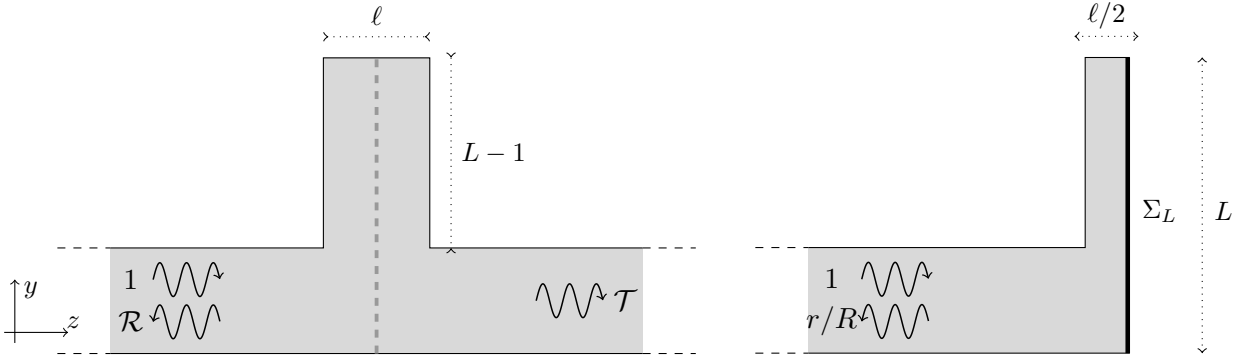

\noindent Set $\ell=\pi/\sqrt{\lambda}$ and for $L>1$, define the waveguide (see Figure \ref{DomainOriginal} left)
\begin{equation}\label{defOriginalDomain}
\mathcal{G}_L:=\{ (y,z)\in (0;1)\times\R\ \cup\  [1;L)\times (-\ell/2;\ell/2)\}. 
\end{equation}
The value $\pi/\sqrt{\lambda}$ for the width of the bounded branch of $\mathcal{G}_L$ is very important in the analysis we develop below. We consider the
dissipationless problem 
\begin{equation}\label{PbInitial}
\begin{array}{|rcll}
\Delta_x v + \lambda v & = & 0 & \mbox{ in }\mathcal{G}_L\\[3pt]
 \partial_nv  & = & 0  & \mbox{ on }\partial\mathcal{G}_L.
\end{array}
\end{equation}

\noindent Let $\chi_l\in\mathscr{C}^{\infty}(\R^2)$ (resp. $\chi_r\in\mathscr{C}^{\infty}(\R^2)$) be a cut-off function equal to one for $z\le -\ell$ (resp. $z\ge\ell$) and to zero for $z\ge -\ell/2$ (resp. $z\le \ell/2$). Problem (\ref{PbInitial}) admits a solution with the expansion
\begin{equation}\label{DefScatteringCoeff}
v= \chi_l\,(w^{+}_{0}+\rcoef\,w^{-}_{0})+\chi_r\,\tcoef\,w^{+}_0+\tilde{v},
\end{equation}
where $\rcoef,\,\tcoef\in\mathbb{C}$ and where $\tilde{v}$ decays exponentially as $O(e^{-\sqrt{\pi^2-\lambda}|z|})$ for $z\to\pm\infty$. The reflection coefficient $\rcoef$ and transmission coefficient $\tcoef$ in (\ref{DefScatteringCoeff}) are uniquely defined and satisfy the relation of conservation of energy $|\rcoef|^2+|\tcoef|^2=1$. They depend on $L$ and below we also write $\rcoef(L)$, $\tcoef(L)$ instead of $\rcoef$, $\tcoef$. For our purpose, we need to study the behaviour of $\rcoef(L)$, $\tcoef(L)$ as functions of $L$. To proceed, we use the symmetry with respect to the $(Oy)$ axis. Define the half-waveguide 
\begin{equation}\label{defHalfWaveguide}
\mathfrak{g}_L:=\{(y,z)\in\mathcal{G}_L\,|\,z<0\}
\end{equation}
(see Figure \ref{DomainOriginal} right). Introduce the problem with Neumann boundary conditions 
\begin{equation}\label{PbChampTotalSym}
\begin{array}{|rcll}
\Delta u +\lambda u & = & 0 & \mbox{ in }\mathfrak{g}_L\\[3pt]
 \partial_nu  & = & 0  & \mbox{ on }\partial\mathfrak{g}_L
\end{array}
\end{equation}
as well as the problem with mixed boundary conditions 
\begin{equation}\label{PbChampTotalAntiSym}
\begin{array}{|rcll}
\Delta U + \lambda U & = & 0 & \mbox{ in }\mathfrak{g}_L\\[3pt]
 \partial_nU  & = & 0  & \mbox{ on }\partial\mathfrak{g}_L\cap\partial\mathcal{G}_L \\[3pt]
U  & = & 0  & \mbox{ on }\Sigma_L:=(0;L)\times\{0\}.
\end{array}
\end{equation}
Problems (\ref{PbChampTotalSym}) and (\ref{PbChampTotalAntiSym}) admit respectively the solutions 
\begin{eqnarray}
\nonumber u &=& w^+_0+\rN\,w^-_0 + \tilde{u},\qquad\hspace{2mm}\mbox{ with }\tilde{u}\in H^1(\mathfrak{g}_L),\\[3pt]
\label{defZetaLanti}U &=& w^+_0+\rD\,w^-_0 + \tilde{U},\qquad\mbox{ with }\tilde{U}\in H^1(\mathfrak{g}_L),
\end{eqnarray} 
where $\rN$, $\rD\in\mathbb{C}$ are uniquely defined. Moreover, due to conservation of energy, one has
\begin{equation}\label{NRJHalfguide}
|\rN|=|\rD|=1.
\end{equation}
Direct inspection shows that if $v$ is a solution of Problem (\ref{PbInitial}) then we have $v(y,z)=(u(y,z)+U(y,z))/2$ in $\mathfrak{g}_L$ and $v(y,z)=(u(y,-z)-U(y,-z))/2$ in $\mathcal{G}_L\setminus\overline{\mathfrak{g}_L}$ (up possibly to a term which is exponentially decaying at $\pm\infty$ if there are trapped modes at the given $\lambda$). We deduce that the scattering coefficients $\rcoef$, $\tcoef$ appearing in the decomposition (\ref{DefScatteringCoeff}) of $v$ are such that
\begin{equation}\label{Formulas}
\rcoef=\frac{\rN+\rD}{2}\qquad\mbox{ and }\qquad \tcoef=\frac{\rN-\rD}{2}.
\end{equation} 
Remark that relations (\ref{Formulas}) rely only on the symmetry of the waveguide, not on the choice (\ref{defOriginalDomain}). Now in the particular geometry considered here, one observes that $U':=w^+_0-w^-_0=i(2/\sqrt{\lambda})^{1/2}\sin(\sqrt{\lambda}z)$ is a solution to (\ref{PbChampTotalSym}) for all $L>1$. Note that this property is based on the fact that $\pi/(2\sqrt{\lambda})$, the width of the bounded branch of $\mathfrak{g}_L$, coincides with the quarter wavelength of the waves $w^{\pm}_0$. By uniqueness of the definition of the coefficient $\rD$ in (\ref{defZetaLanti}), we deduce that 
\begin{equation}\label{PartRelation}
\rD=\rD(L)=-1\qquad\mbox{ for all $L>1$}. 
\end{equation}
On the other hand, it is showed in \cite[\S3.2]{ChNP18} that as $L\to+\infty$, $\rN(L)$ runs continuously on the unit circle $\mathscr{C}(0,1)$ (remember the constraint (\ref{NRJHalfguide})) and that we have $|\rN(L)-\rN^{\mrm{asy}}(L)|\le C\,e^{-c\,L}$ for certain $c,\,C>0$ where $L\mapsto \rN^{\mrm{asy}}(L)$ runs periodically on $\mathscr{C}(0,1)$ with a period of $\pi/\sqrt{\lambda}$. Therefore from (\ref{Formulas}), we infer that as $L\to+\infty$, $\rcoef(L)$ runs continuously on $\mathscr{C}(-1/2,1/2)$, almost periodically with a period of $\pi/\sqrt{\lambda}$. This shows that by tuning $L$, we can get any reflection coefficient of $\mathscr{C}(-1/2,1/2)$. For our analysis below, we need more and wish to attain any point of the unit disk. To proceed, let us define the shifted waveguide 
\begin{equation}\label{defHalfWaveguideShift}
\mathcal{G}^{\sigma}_L:=\{(y,z)\,|\,(y,z-\sigma)\in\mathcal{G}_L\}.
\end{equation}
Then Problem (\ref{PbInitial}) set in $\mathcal{G}^{\sigma}_L$ admits a solution with the expansion 
\begin{equation}\label{DefScatteringCoeffShift}
v^{\sigma}= \chi_l\,(w^{+}_{0}+\rcoef^{\sigma}\,w^{-}_{0})+\chi_r\,\tcoef^{\sigma}\,w^{+}_0+\tilde{v}^{\sigma},
\end{equation}
where $\tilde{v}^{\sigma}$ decay exponentially at infinity. Clearly we have $v^{\sigma}(y,z)= e^{i\sqrt{\lambda}\sigma}v(y,z-\sigma)$, possibly up to some exponentially decaying terms if trapped modes exist, which shows in particular that $\rcoef^{\sigma}=e^{2i\sqrt{\lambda}\sigma}\rcoef$. Therefore this proves that by choosing carefully $L$ and $\sigma$, for $\rcoef^{\sigma}(L)$ we can get any value in the unit disk.

\subsection{Association with the dissipative waveguide}

Let us come back to our dissipative waveguide for which  $\Om$, $\mathcal{O}$, $\eta>0$, and so $\mathbb{S}^\eta$, are given once for all. We want to modify its geometry to create a new device where the reflection coefficient is quasi null. Let us fix $L$ and $\sigma>0$ in (\ref{defHalfWaveguideShift}) such that $\rcoef^{\sigma}$, the reflection coefficient appearing in (\ref{DefScatteringCoeffShift}), satisfies $\overline{\rcoef^{\sigma}}=\mathbb{S}^\eta$. For $\kappa\in\N$, then define the domain 
\begin{equation}\label{DefShitedDomain}
\Om^\kappa:=\Om\cup\mathcal{L}^\kappa\qquad\mbox{ with }\mathcal{L}^\kappa:=[1;L)\times(-\ell/2+2\kappa\pi/\sqrt{\lambda}+\sigma;\ell/2+2\kappa\pi/\sqrt{\lambda}+\sigma).
\end{equation}
The problem 
\[
\begin{array}{|rlcl}
\Delta_x u^\kappa+\lambda \big( 1 + i \eta\,b \big) u^\kappa  &=& 0 & \mbox{ in }\Omega^\kappa\\[3pt]
\partial_nu^\kappa &=& 0  & \mbox{ on }\partial\Omega^\kappa
\end{array}
\]
admits a solution of the form
\[
u^\kappa = w_0^- + R^\kappa\,w_0^+   + \tilde{u}^\kappa,
\]
where as usual $\tilde{u}^\kappa$ decays exponentially at infinity and where $R^\kappa$ is the reflection coefficient of interest. Let us explain why by taking $\kappa\in\N$ large, we can make $R^\kappa$ as small as we wish. To proceed, roughly speaking, we show that when $\kappa\in\N$ tends to infinity, $u^\kappa$ gets closer and closer to some approximation $\tilde{u}^\kappa$ for which the corresponding reflection coefficient is null. Introduce $\zeta$ a smooth cut-off function such that $\zeta=1$ for $z\le 0$ and $\zeta=0$ for $z\ge1$. Then define $v^\kappa$, $\zeta^\kappa$ such that $v^\kappa(y,z)=v^\sigma(y,z-2\kappa\pi/\sqrt{\lambda})$, $\zeta^\kappa(y,z)=\zeta(y,z-\kappa\pi/\sqrt{\lambda})$ and set 
\[
\hat{u}^\kappa:=(\zeta^\kappa u^\eta+(1-\zeta^\kappa)\overline{v^\kappa}\,)/\overline{\tcoef^\sigma}
\]
(here $u^\eta$ is the one introduced in (\ref{08})). Note that since $\eta>0$, there holds $|\rcoef^\sigma|=|\mathbb{S}^\eta|<1$ and so $\tcoef^\sigma\ne0$. At this stage, we need to comment on this choice for $\hat{u}^\kappa$. First observe that in the transition region $\kappa\pi/\sqrt{\lambda}\le z \le \kappa\pi/\sqrt{\lambda}+1$, the main parts of $u^\eta$ and $\overline{v^\kappa}$ are the same and coincide with $w^{+}_{0}+\mathbb{S}^\eta w^{-}_{0}$. This is due to the fact that the behaviour of $\hat{u}^\kappa$ at $+\infty$ matches the one of $\overline{v^\kappa}$ at $-\infty$. Second at $+\infty$, we have $\hat{u}^\kappa=w^{-}_{0}+\dots$ where ellipsis stand for exponentially decaying terms (the reflection coefficient for $\hat{u}^\kappa$ is null). Therefore we find that $E^\kappa:=u^\kappa-\hat{u}^\kappa$ is an outgoing function which solves 
\begin{equation}\label{pbKappa}
\begin{array}{|rlcl}
\Delta_x E^\kappa+\lambda \big( 1 + i \eta\,b \big) E^\kappa  &=& F^\kappa & \mbox{ in }\Omega^\kappa\\[3pt]
\partial_nE^\kappa &=& 0  & \mbox{ on }\partial\Omega^\kappa,
\end{array}
\end{equation}
where the discrepancy $F^\kappa$ is given by 
\[
F^\kappa=-\left(2\nabla \zeta^\kappa\cdot\nabla (u^\eta-\overline{v^\kappa})+\Delta \zeta^\kappa(u^\eta -\overline{v^\kappa})\right)/\overline{\tcoef^\sigma}.
\]
Using that $u^\eta$ and $\overline{v^\kappa}$ have the same propagating behaviours on the supports of $\nabla \zeta^\kappa$, $\Delta \zeta^\kappa$, one finds that the $L^2$ norm of $F^\kappa$ tends to zero as $\kappa(\in\N)\to+\infty$. From an additional result of uniform stability for the solution of Problem (\ref{pbKappa}) with respect to the source term as $\kappa\to+\infty$, finally we deduce that the scattering coefficient of $E^\kappa$, which is equal to $R^\kappa$, tends to zero as $\kappa(\in\N)\to+\infty$. We emphasize that this stability result is not trivial to establish but can be derived by adapting the approach presented in \cite[Ch.\,5,\,\S5.6,\,Thm. 5.6.3]{MaNaPl}. This shows that we can create dissipative guides for which the reflection coefficient in monomode regime is as small as we wish.

\subsection{Numerical illustrations and comments}

Let us give numerical examples of perfect absorbers to illustrate the method. The tools are the same as in Section \ref{SectionNumerics}. In practice, for given $\Om$, $\mathcal{O}$, $\eta>0$, we start by computing $\mathbb{S}^\eta$. Then we find via an explicit calculus values of $\sigma$ such that the circle $e^{-2i\sqrt{\lambda}\sigma}\mathscr{C}(-1/2,1/2)$ in the complex plane passes through $\mathbb{S}^\eta$. If $\mathbb{S}^\eta=\alpha e^{i\beta}$ for some $\alpha\in[0;1)$ and $\beta\in[0;2\pi)$, one solution is to take
\[
\sigma=\cfrac{\arccos(-\alpha)-\beta}{2\sqrt{\lambda}}+k\pi/\sqrt{\lambda},\qquad k\in\mathbb{Z}.
\]
Then we fix $\kappa\in\N$ large enough in the definition of $\Om^\kappa$ in (\ref{DefShitedDomain}) so that the resonator  of $\Om$ and the bounded branch $\mathcal{L}^\kappa$ are separated. Note that since the evanescent modes decay exponentially, we do not need to take a large $\kappa$. Finally we vary $L$ and identify the interesting geometries by looking at the peaks of the curve $L\mapsto -\ln|R^\kappa(L)|$.\\

\begin{figure}[!ht]
\centering
\includegraphics[width=0.45\textwidth]{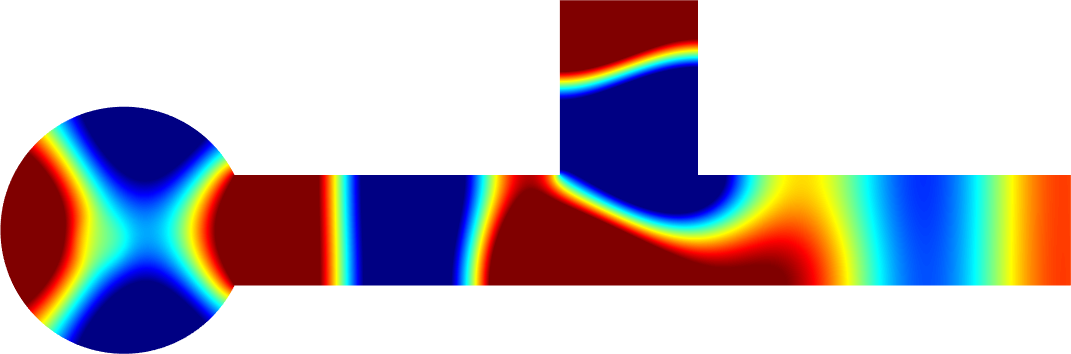}\qquad\includegraphics[width=0.45\textwidth]{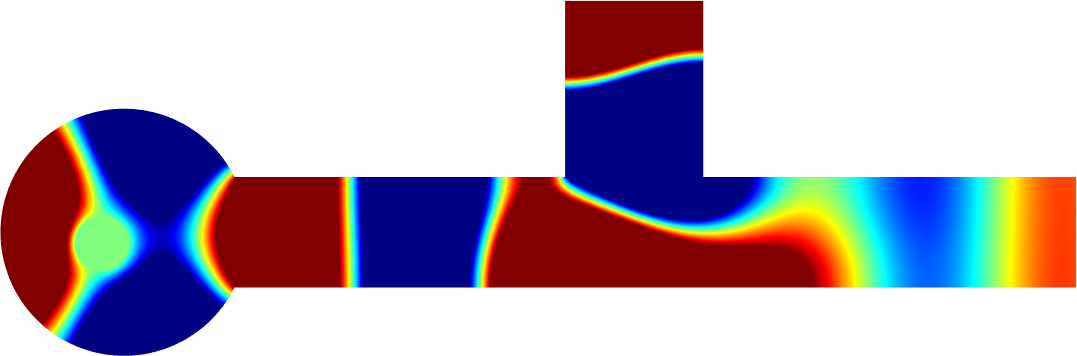}\\[2pt]
\includegraphics[width=0.45\textwidth]{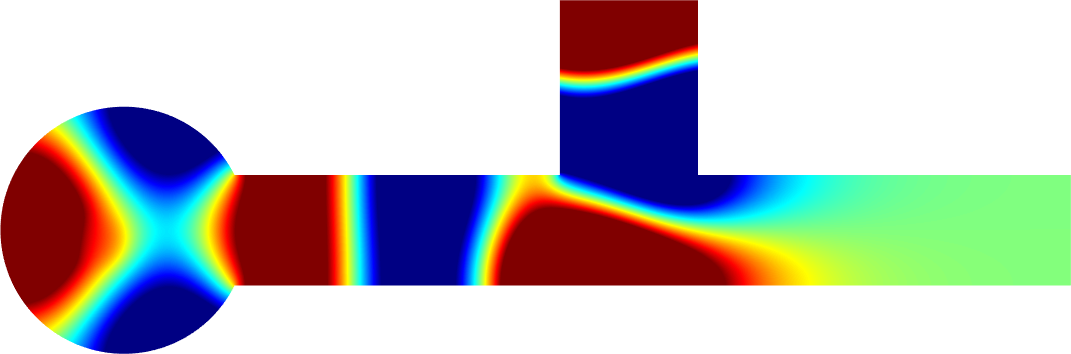}\qquad\includegraphics[width=0.45\textwidth]{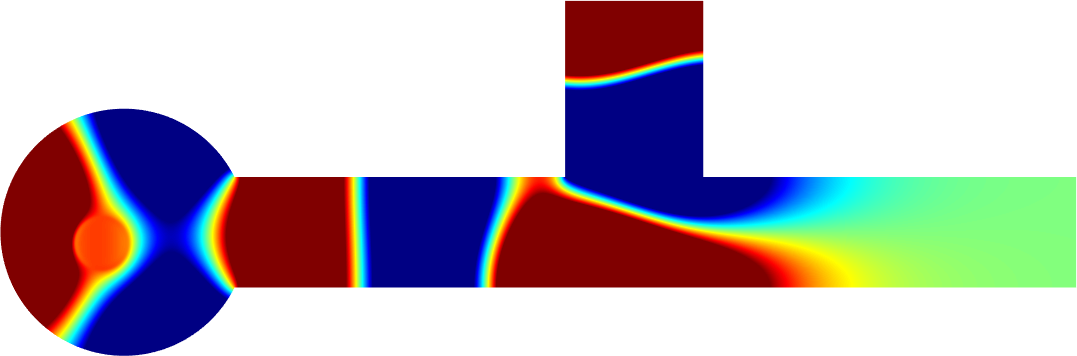}
\caption{$\Re e\,u^\kappa$ (first line) and $\Re e\,(u^\kappa-w_0^-)$ (second line) in two situations where $R^\kappa\approx0$. We take $\eta=10$ for the left pictures and $\eta=1000$ for the right pictures.\label{LargeWidth}}
\end{figure}

\begin{figure}[!ht]
\centering
\includegraphics[trim={1.6cm 1.2cm 1.2cm 0.8cm},clip,width=0.42\textwidth]{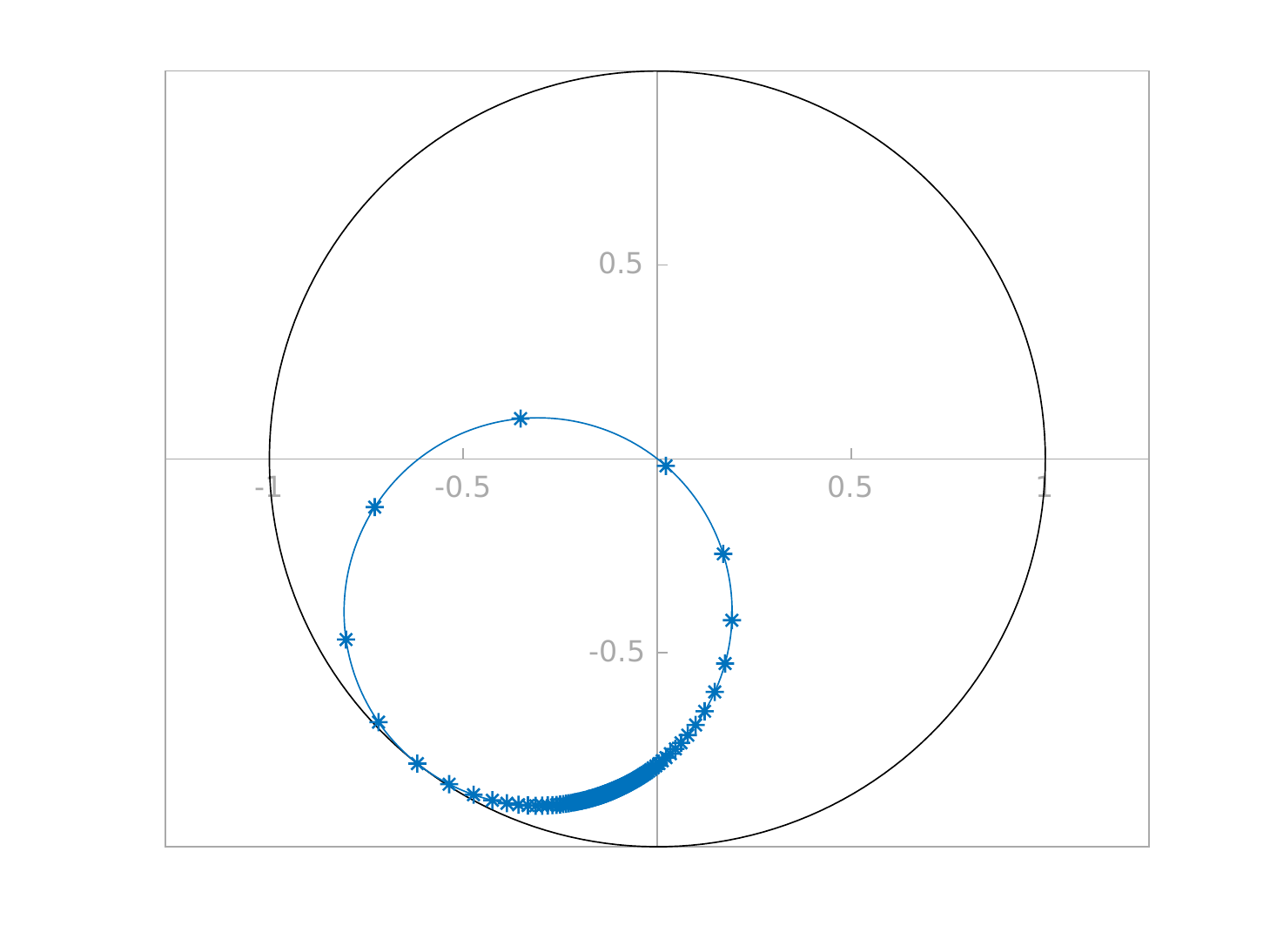}\qquad\includegraphics[trim={1.6cm 1.2cm 1.2cm 0.8cm},clip,width=0.42\textwidth]{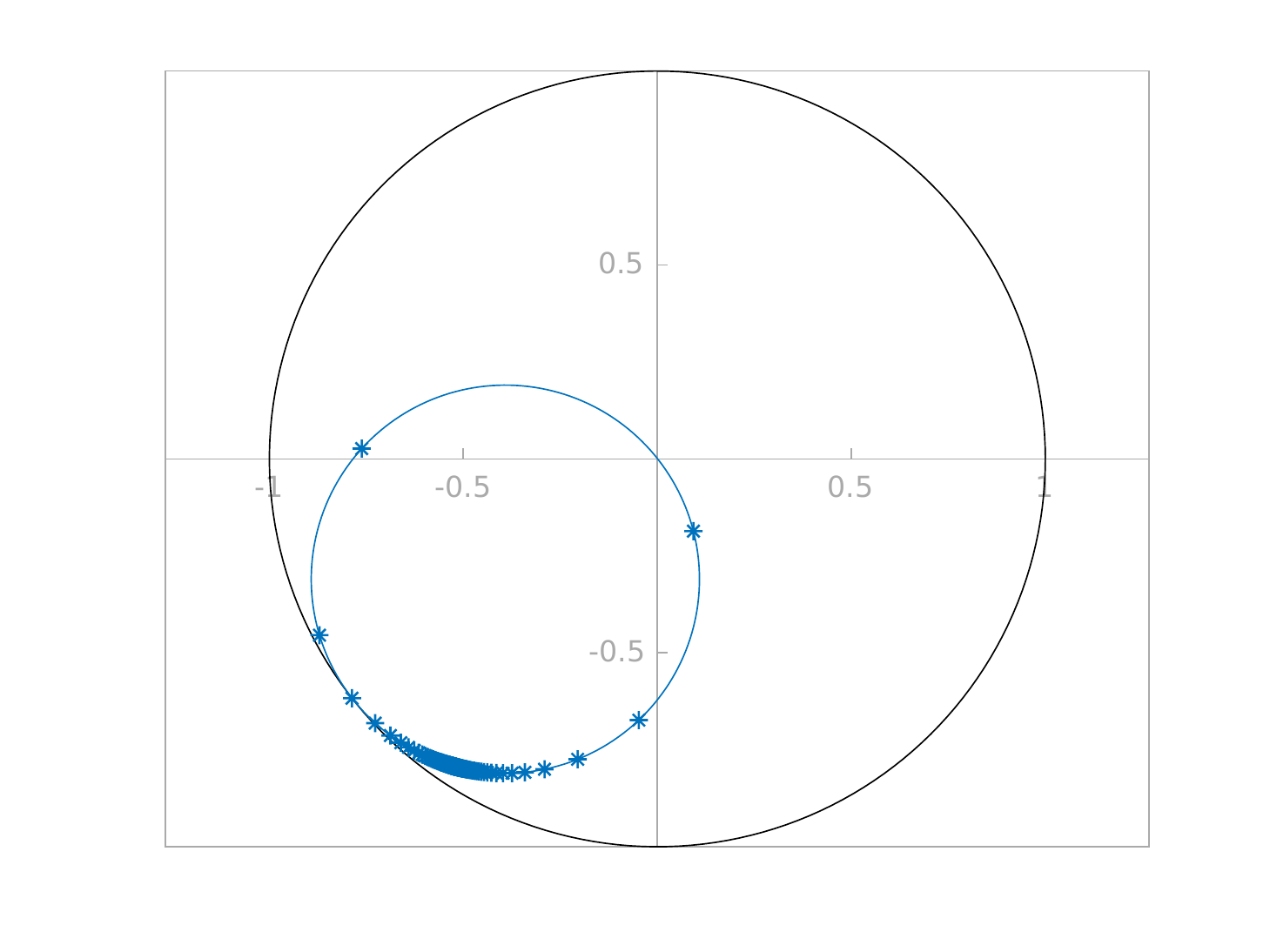}
\caption{Curve $L\mapsto R^\kappa(L)$ in the complex plane for $L\in[0.5;1.75]$. The geometry is the same as in Figure  \ref{LargeWidth}. For the left (resp. right) picture, we take $\eta=10$ (resp. $\eta=1000$).  The black circle represents the unit circle.\label{LargeScattering}}
\end{figure}

\noindent In Figure \ref{LargeWidth}, we display two situations where $R^\kappa\approx0$. We indeed observe that the scattering field $u^\kappa-w_0^-$ is approximately exponentially decaying at infinity. Each column corresponds to a different value of $\eta$. For $\eta=1000$, the dissipative inclusion behaves mainly as a Dirichlet obstacle (see again Theorem \ref{thT2}) so that $|\mathbb{S}^\eta|$ is closed to one. In this situation, it is harder to remove the reflection. In our approach, we observe that the tuning of $L$ becomes more delicate, the device is less robust to geometric perturbations. In Figure \ref{LargeScattering}, we present the curves $L\mapsto R^\kappa(L)$ in the complex plane. The lost of robustness for large $\eta$ translates into a rapid passage of $L\mapsto R^\kappa(L)$ through zero.

\noindent In Figures \ref{ThinWidth}, \ref{ThinScattering}, we use a geometry different from the one introduced in \S\ref{paragraphAnyR} to get a dissipationless waveguide with a prescribed reflection coefficient. More precisely, we replace the bounded branch of width $\ell$ appearing in the definition of $\mathcal{G}_L$ (see (\ref{defOriginalDomain})) by a thin ligament. It is shown in \cite[\S4.1]{ChHN22} that when the length of this thin ligament varies slightly around its resonance lengths $\pi(m+1/2)/\sqrt{\lambda}$, $m\in\N$, the corresponding reflection coefficient approximately runs on $\mathscr{C}(-1/2,1/2)$. Shifting the geometry in the $z$ direction as we did in \S\ref{paragraphAnyR} then allows one to obtain approximately any reflection coefficient in the unit disk. Let us mention that a thin ligament of fixed length is resonant at a given wavenumber and generates large reflections only for tight bands of wavenumbers. Due to this property, a collection of ligaments of different lengths act separately at first order. In other words, their actions decouple. This is interesting to construct devices which perfectly absorb the energy of incident plane waves 
at different wavenumbers $\lambda_1,\dots,\lambda_N\in (0;\pi^2)$. Note that this would not be doable with the bounded branches of \S\ref{paragraphAnyR} whose actions do not decouple. Finally let us mention that the two mechanisms we propose above to get any reflection of $\mathscr{C}(-1/2,1/2)$ need to be adapted to consider Dirichlet boundary conditions. The rest of the analysis would be completely similar but this first step requires to be studied (for example the thin ligament would have to be replaced by another type of resonator).

\begin{figure}[!ht]
\centering
\includegraphics[width=0.45\textwidth]{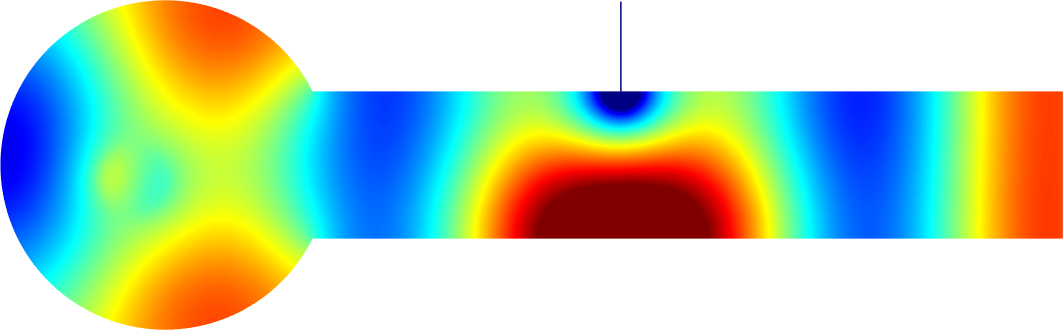}\qquad\includegraphics[width=0.45\textwidth]{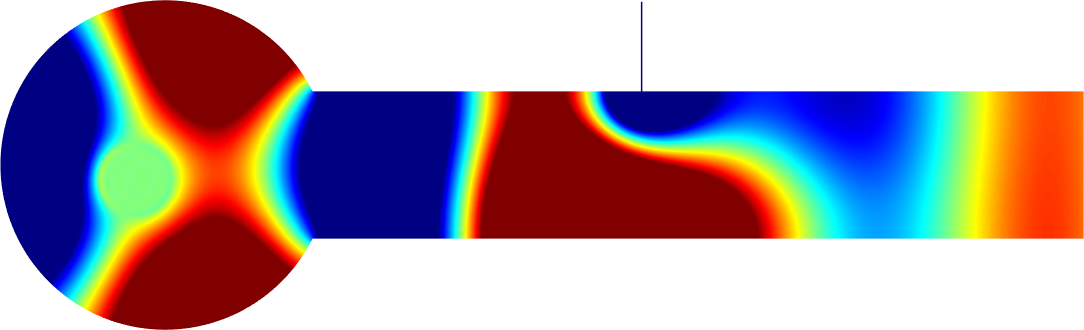}\\[2pt]
\includegraphics[width=0.45\textwidth]{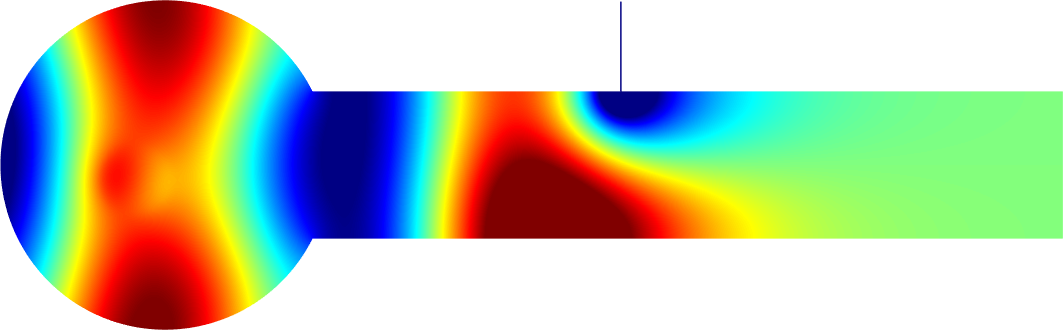}\qquad\includegraphics[width=0.45\textwidth]{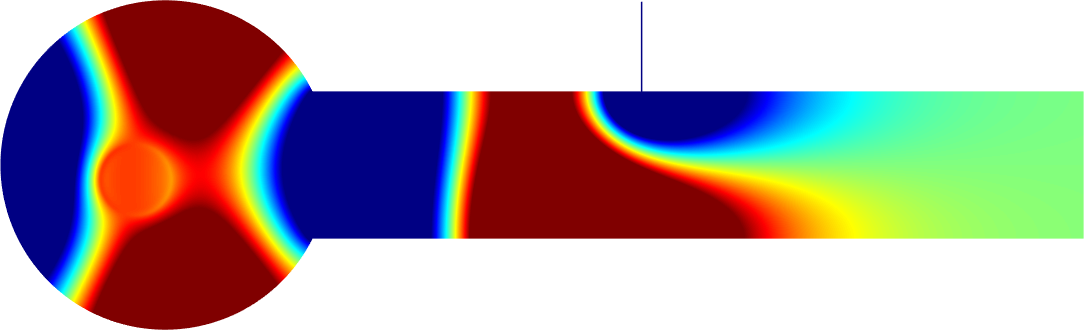}
\caption{$\Re e\,u^\kappa$ (first line) and $\Re e\,(u^\kappa-w_0^-)$ (second line) in two situations where $R^\kappa\approx0$. We take $\eta=10$ for the left pictures and $\eta=1000$ for the right pictures.\label{ThinWidth}}
\end{figure}

\begin{figure}[!ht]
\centering
\includegraphics[trim={1.6cm 1.2cm 1.2cm 0.8cm},clip,width=0.4\textwidth]{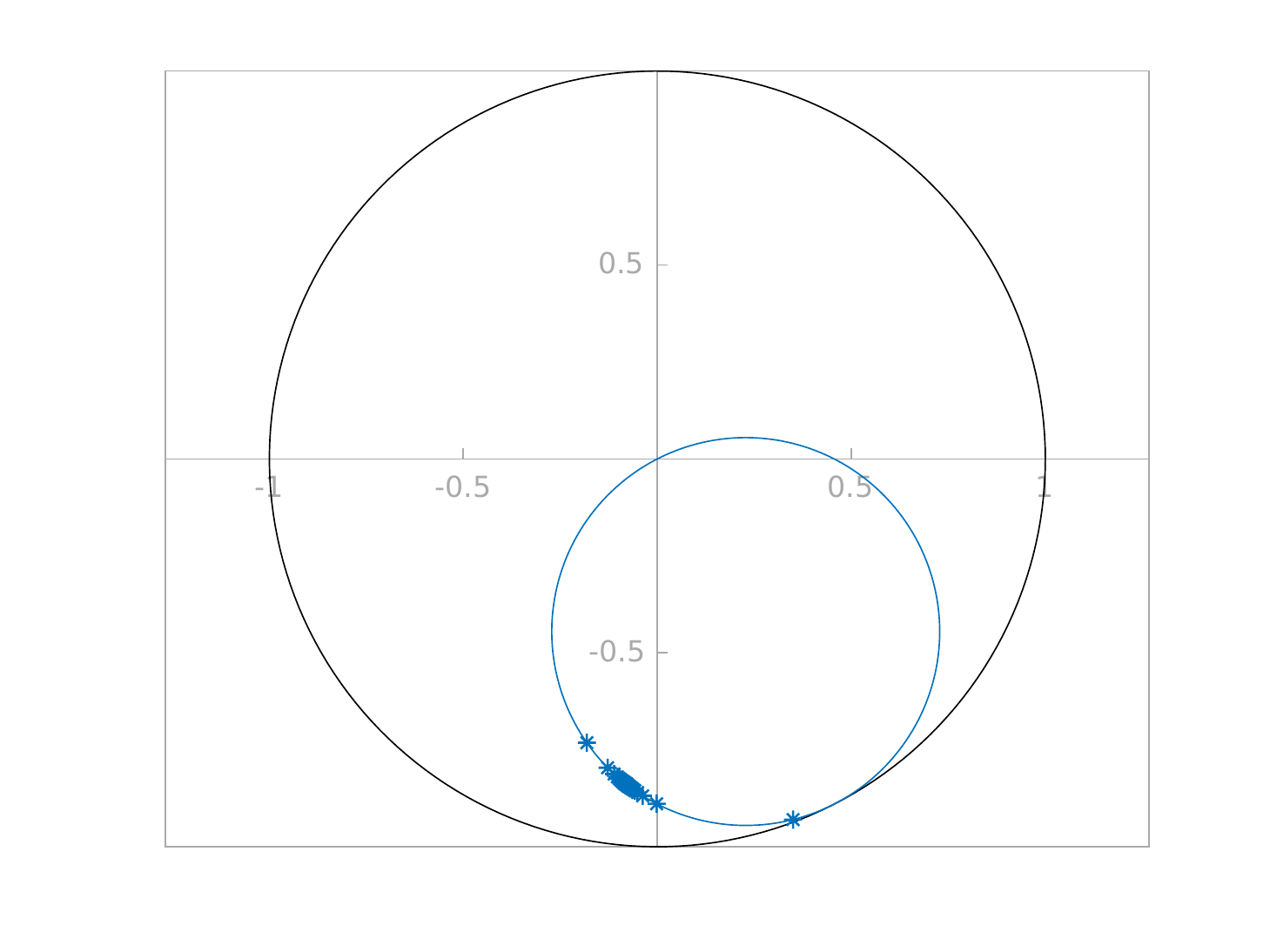}\qquad\includegraphics[trim={1.6cm 1.2cm 1.2cm 0.8cm},clip,width=0.42\textwidth]{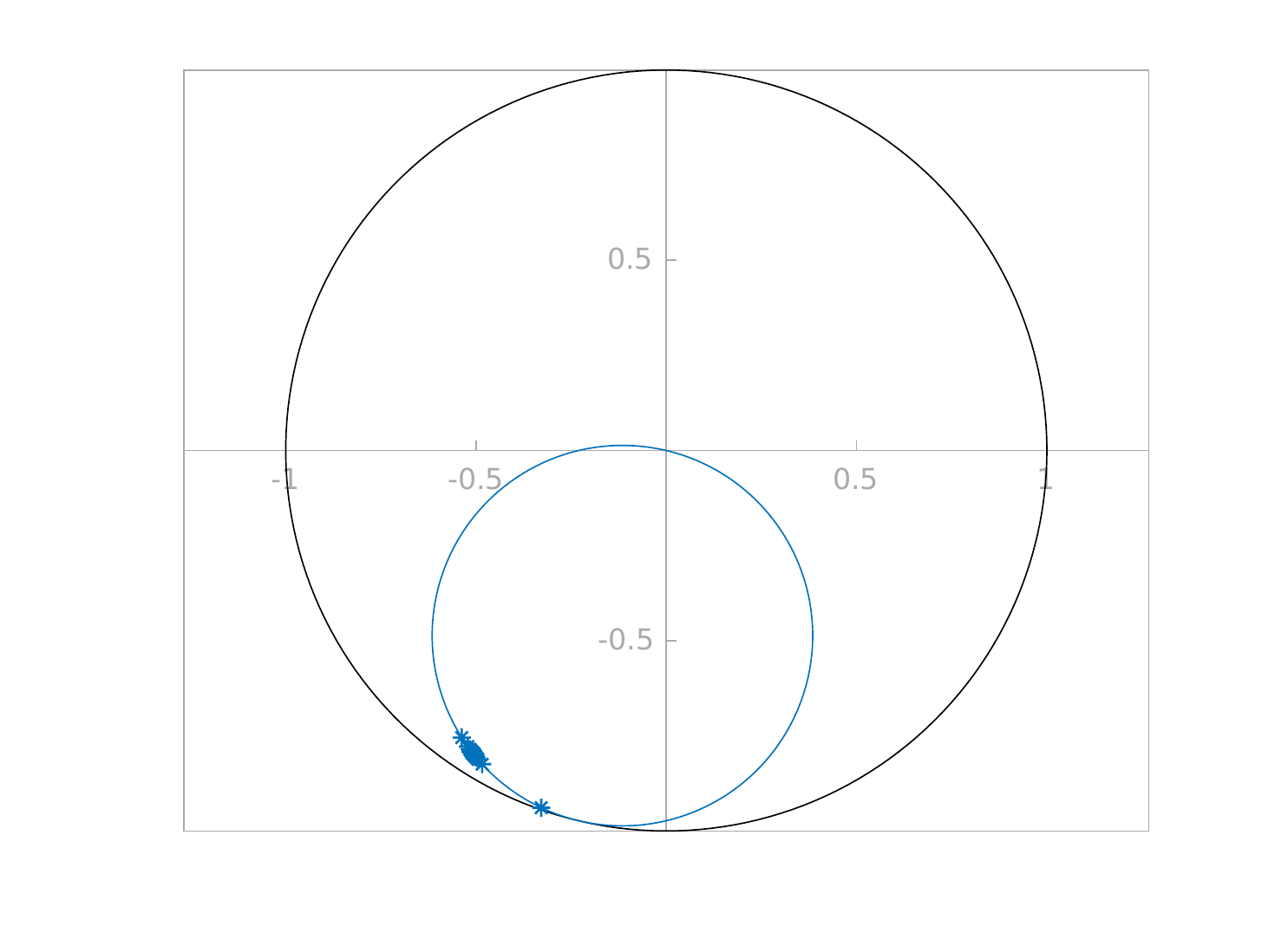}
\caption{Curve $L\mapsto R^\kappa(L)$ in the complex plane for $L\in[0.125;1.125]$. The geometry is the same as in Figure  \ref{ThinWidth}. For the left (resp. right) picture, we take $\eta=10$ (resp. $\eta=1000$).  The black circle represents the unit circle.\label{ThinScattering}}
\end{figure}

\newpage 

\section{Appendix}  \label{sec9}
In this appendix, we give the proof of Theorem \ref{thT3}. To proceed, we start by presenting a few intermediate results.\\
\newline
$1^\circ$. Hardy-type inequality. For $U \in \mathscr{C}_0^\infty([0;+\infty))$
and $\varepsilon > 0$, we can write
\[
\begin{array}{rcl}
\dsp\int_{0}^{+ \infty} ( r + \varepsilon)^{-2} |U(r)|^2 \,dr & = &
\dsp\int_{0}^{+ \infty} ( r + \varepsilon)^{-2} |U(0)|^2 \,dr 
+ 2 \int_{0}^{+ \infty} ( r + \varepsilon)^{-2}
\int_{0}^r \frac{dU(t)}{dt}U(t)\,dt dr \\[10pt]
& \le &\dsp \frac1{	\varepsilon} |U(0)|^2 +
2 \int_{0}^{+ \infty}  \Big| \frac{dU(t)}{dt} \Big| |U(t)|
\int_{t}^{+ \infty}  ( r + \varepsilon)^{-2} \,dr dt \\[10pt]
& \le &\dsp \frac1{\varepsilon} |U(0)|^2 +
2\int_{0}^{+ \infty}  \Big| \frac{dU(t)}{dt} \Big|(t+ \varepsilon)^{-1} 
|U(t)|\,dt \\[10pt]
& \le &\dsp \frac1{\varepsilon} |U(0)|^2 +
2 \int_{0}^{+ \infty}  \Big| \frac{dU(t)}{dt} \Big|^2  \,dt 
+ \frac12 \int_{0}^{+ \infty}  (t+ \varepsilon)^{-2}  |U(t)|^2\,dt .
\end{array}
\]
Thus, we obtain
\begin{equation}\label{a1} 
\int_{0}^{+ \infty} ( r + \varepsilon)^{-2} |U(r)|^2 \,dr 
\leq \frac2{\varepsilon} |U(0)|^2 +
4 \int_{0}^{+ \infty}  \Big| \frac{dU(r) }{dr} \Big|^2  \,dr .
\end{equation}
This inequality holds also for $\varepsilon= 0$ in the case $U(0)=0$. Then it coincides with the traditional Hardy inequality.\\
\newline
$2^\circ$. Auxiliary problems. The analysis of Problem (\ref{29}) with mixed boundary conditions leads to consider the following variational formulation
\bea\label{a2}
\begin{array}{|l}
\,\mbox{Find }u^\infty\in\mathring{W}^{\mrm{out}}(\Omega_\bullet)\mbox{ such that }\\[4pt]
\,( \nabla_x u^\infty, \nabla_x \varphi )_{\Omega_\bullet} -
\lambda ( u^\infty, \varphi )_{\Omega_\bullet} 
= F (\varphi),\qquad \forall\varphi \in \mathscr{C}^\infty_0(\Omega_\bullet), 
\end{array}
\eea 
where $F \in \mathring{W}^1_{-\beta}(\Omega_\bullet)^* $. Here $\mathring{W}^1_{\pm\beta}(\Omega_\bullet)$ (resp. $\mathring{W}^{\mrm{out}}(\Omega_\bullet)$) denotes the subspace of $W^1_{\pm\beta}(\Omega_\bullet)$ (resp. $W^{\mrm{out}}(\Omega_\bullet)$) made of functions vanishing on $\partial \mathcal{O}$. The norms in these spaces are defined as in (\ref{160}) (resp. (\ref{18})) with $\Om$ replaced by $\Om_\bullet$. In particular, for $u^\infty = \chi\sum_{j=0}^{J-1} a_j w_j^+  + \tilde{u}^\infty$, we have 
\bea\label{a3}
\Vert u^\infty ; \mathring{W}^{\mrm{out}}(\Omega_\bullet) \Vert = \big( |a|^2 
+ \Vert \tilde{u}^\infty ; W_\beta^1(\Omega_\bullet) \Vert^2 \big)^{1/2}.
\eea 
The Hardy inequality \eqref{a1} with $\varepsilon=0$, integrated 
over the surface $\partial \mathcal{O}$, together with the Dirichlet condition $u^\infty = \tilde{u}^\infty = 0$ on $\partial \mathcal{O}$, imply that the norm \eqref{a3} is equivalent to the norm 
\bea\label{a4}
\big\Vert u^\infty ; \mathring{W}^{\mrm{out}}_\gamma(\Omega_\bullet)\big\Vert=\big( |a|^2 + \Vert \varrho^\gamma e^{\beta z } 
\nabla_x \tilde{u}^\infty ; L^2(\Omega_\bullet) \Vert^2 
+ \Vert \varrho^{\gamma -1} e^{\beta z } 
\tilde{u}^\infty ; L^2(\Omega_\bullet) \Vert^2  
\big)^{1/2} 
\eea 
for $\gamma=0$. Here $\varrho$ is the distance function introduced in \eqref{47}. For the analysis below, we also introduce the space $\mathring{W}_{\beta,\gamma}^{1}(\Omega_\bullet)$ defined as the completion of $\mathscr{C}^\infty_0(\Om_\bullet)$ for the norm
\bea\label{a4bis}
\big\Vert \varphi ; \mathring{W}^{1}_{\beta,\gamma}(\Omega_\bullet)\big\Vert=\big(\Vert \varrho^\gamma e^{\beta z } 
\nabla_x \varphi ; L^2(\Omega_\bullet) \Vert^2 
+ \Vert \varrho^{\gamma -1} e^{\beta z } 
\varphi ; L^2(\Omega_\bullet) \Vert^2  
\big)^{1/2}. 
\eea 
In the following, we shall need to work with the following problem
\bea\label{a5}
\begin{array}{|rccl}
\multicolumn{4}{|l}{\,\mbox{Find }v^\eta\in \mathring{H}^1(\Xi)\mbox{ such that }}\\[3pt]
\, - \Delta_x v^\eta  - \lambda( 1+ i\eta\,b ) v^\eta &=& f & \mbox{ in }\Xi, 
\end{array}
\eea 
where $\Xi$ is a bounded domain such that $\overline \Xi \subset 
\Omega$ and $\overline{\mathcal{O}}  \subset \Xi$. Additionally we will choose the ``width'' of the annulus $\Xi_\bullet = \Xi\setminus\overline{\mathcal{O}}$ small enough so that the first eigenvalue $\mu_0$ of the 
mixed boundary value problem
\bea\label{a6}
\begin{array}{|l}
- \Delta_x v = \mu v\mbox{ in }\Xi_\bullet\\ [3pt]
\,v = 0 \mbox{ on }\partial\Xi,\qquad\partial_n v = 0\mbox{ on }\partial\mathcal{O},  
\end{array}
\eea
is larger than $\lambda$ (note that this is always possible). In (\ref{a5}),  $\mathring{H}^1(\Xi)$ denotes the subspace of $H^1(\Xi)$ made of functions vanishing on $\partial\Xi$. The variational formulation of (\ref{a5}) writes 
\bea\label{a7}
\begin{array}{|rccl}
\multicolumn{4}{|l}{\,\mbox{Find }v^\eta\in \mathring{H}^1(\Xi)\mbox{ such that }}\\[3pt]
\,\mathfrak{a}_\Xi(v^\eta,\phi)=F(\phi),\qquad\forall \phi\in \mathring{H}^1(\Xi),
\end{array}
\eea
where $\mathfrak{a}_\Xi(v^\eta,\phi)=(\nabla_x v^\eta, \nabla_x \phi )_\Xi -\lambda( v^\eta, \phi )_\Xi- i\eta ( b v^\eta, \phi )_\mathcal{O} $ and $F \in \mathring{H}^1(\Xi)^* $. Using the Lax-Milgram lemma, one can show that Problem (\ref{a7}) is uniquely solvable. In particular, to show the coercivity of $\mathfrak{a}_\Xi(\cdot,\cdot)$, one can write, for all $v\in \mathring{H}^1(\Xi)$, 
\[
\Re e\, \mathfrak{a}_\Xi(v,v) = \Vert \nabla_x v ; L^2(\Xi) \Vert^2 - 
\lambda \Vert v ; L^2(\Xi) \Vert^2,\qquad \Im m \,  \mathfrak{a}_\Xi(v,v) \geq  c_b\,\eta\,\Vert  v ; L^2(\mathcal{O}) \Vert^2,\quad c_b > 0,
\]
and use the fact that $\Vert \nabla_x v ; L^2(\Xi_\bullet) \Vert^2  \geq \mu_0 \Vert  ;L^2(\Xi_\bullet) \Vert^2$ with $\mu_0 > \lambda$, to get 
\begin{equation}\label{estima1}
|\mathfrak{a}_\Xi(v,v)| \geq c\,\big( \Vert  v; H^1(\Xi) \Vert^2 
+ \eta \Vert v ; L^2(\mathcal{O}) \Vert^2 \big),\qquad  \forall\eta \ge \eta_0 > 0.
\end{equation}
This estimate together with the trace inequality \eqref{55} imply
\bea\label{a9}
|\mathfrak{a}_\Xi(v,v) | \ge c\,\sqrt{\eta}\,\Vert v ; L^2(\partial \mathcal{O} ) \Vert^2 .
\eea
Using (\ref{a9}) in the Hardy inequality \eqref{a1} with $\eps= \eta^{-1/2}  > 0$, we obtain 
\bea\label{a10}
|\mathfrak{a}_\Xi(v,v) | \ge c\,\Vert (\varrho + \eta^{-1/2})^{-1}v ; L^2(\Xi_\bullet ) \Vert^2 .
\eea
Finally, combining the above relations leads to the following estimate for the solution $v^\eta\in \mathring{H}^1(\Xi)$ of the problem \eqref{a7}
\bea\label{a11}
\big\Vert v^\eta ; \mathring{H}_{\eta,0}^{1} (\Xi) \big\Vert 
\leq c\,\big\Vert F; \mathring{H}_{\eta,0}^{1} (\Xi)^*  \big\Vert, 
\eea
where the constant $c>0$ is independent of the functional $F\in \mathring{H}^1(\Xi)^*$ and of $ \eta \geq \eta_0 > 0$. Here the space 
$\mathring{H}_{\eta,\gamma}^{1} (\Xi)$, with the new weight index $\gamma \in \bbR$, stands for the Sobolev space $\mathring{H}^1(\Xi)$ endowed with the norm
\begin{equation}\label{a12}
\begin{array}{rcl}
\big\Vert \varphi ; \mathring{H}_{\eta,\gamma}^{1} (\Xi) \big\Vert &=& \big( 
\ \Vert (\varrho + \eta^{-1/2})^\gamma\nabla_x \varphi ; L^2(\Xi_\bullet ) \Vert^2
+ \eta^{- \gamma} \Vert \nabla_x \varphi ; L^2(\mathcal{O} ) \Vert^2  \\[3pt]
& &\phantom{(c}\Vert (\varrho + \eta^{-1/2})^{\gamma-1}\varphi ; L^2(\Xi_\bullet ) \Vert^2
+ \eta^{1- \gamma} \Vert \varphi ; L^2(\mathcal{O} ) \Vert^2 \big)^{1/2}.
\end{array} 
\end{equation}

\noindent $3^\circ$. Varying the weight index. Let us show that the estimate (\ref{a11}) is still valid for an open interval of indices $\gamma$ centered at zero. Define the weight functions $R_{\pm \gamma}^\eta$ such that 
\begin{equation}\label{a14}
R_{\pm \gamma}^\eta (x) = 
\begin{array}{|ll}
\,(\varrho(x) + \eta^{-1/2} )^{\pm \gamma} & \mbox{ in }\Xi_\bullet \\
\,\eta^{\mp \gamma/2} & \mbox{ in }\mathcal{O}. 
\end{array}    
\end{equation} 
Note that $R_{\pm \gamma}^\eta$ are piecewise smooth functions which are continuous at $\partial\mathcal{O}$. Set 
\bea\label{a13}
V^\eta = R_{\gamma}^\eta v^\eta, \qquad \Phi = R_{-\gamma}^\eta \phi.  
\eea
Observe that $V^\eta$, $\Phi$ belong to $\mathring{H}^1(\Xi)$ and that we have $\nabla_x R_{\pm \gamma}^\eta=0$ in $\mathcal{O}$ as well as
\begin{equation}\label{EstimDeriv}
\begin{array}{rcl}
|\nabla_x R_{\pm \gamma}^\eta (x) &\leq& |\gamma| \,\big( \varrho(x) + \eta^{-1/2} \big)^{ \pm \gamma -1} |\nabla_x \varrho(x) |  \\[5pt]
& \le & C |\gamma|\, R_{\pm \gamma}^\eta (x) \,\big( \varrho(x) + \eta^{-1/2} \big)^{-1}\qquad\mbox{ in }\Xi_\bullet.   
\end{array}
\end{equation}
Using this estimate and the fact that there holds $v^\eta = R_{-\gamma}^\eta V^\eta$, one establishes the  inequalities
\bea\label{a0}
\big\Vert v^\eta ; \mathring{H}_{\eta,\gamma}^{1} (\Xi) \big\Vert \leq 
c_\gamma \big\Vert V^\eta ; \mathring{H}_{\eta,0}^{1} (\Xi) \big\Vert
\leq C_\gamma \big\Vert v^\eta ; \mathring{H}_{\eta,\gamma}^{1} (\Xi) \big\Vert,
\eea
where the constants $c_\gamma$, $C_\gamma$ are independent of
$\eta$ and $v^\eta$. Making the substitution $v^\eta = R_{\gamma}^\eta V^\eta$, $\phi=R_{-\gamma}^\eta\Phi$ in the integral identity \eqref{a7}, we get
\[
\mathfrak{a}_\Xi(V^\eta,\Phi) +\tilde{\mathfrak{a}}_\Xi (V^\eta, \Phi)
= F( R_{-\gamma}^\eta \Phi),\qquad\forall \Phi\in 
\mathring{H}^1 (\Xi), 
\]
with 
\[
\tilde{\mathfrak{a}}_\Xi(V^\eta,\Phi) =\big( \nabla_x V^\eta, 
\Phi R_{-\gamma}^\eta \nabla_x R_{\gamma}^\eta \big)_{\Xi_\bullet}
+ \big( V^\eta  R_{\gamma}^\eta \nabla_x R_{-\gamma}^\eta, \nabla_x \Phi 
\big)_{\Xi_\bullet} +\big( V^\eta R_{\gamma}^\eta \nabla_x R_{-\gamma}^\eta,\Phi R_{-\gamma}^\eta \nabla_x R_{\gamma}^\eta \big)_{\Xi_\bullet}.
\]
Exploiting (\ref{EstimDeriv}) as well as (\ref{estima1})--(\ref{a10}), one finds
\[
|\tilde{\mathfrak{a}}_\Xi(V^\eta,\Phi)| \le C\,|\gamma|\,\big\Vert V^\eta ; \mathring{H}_{\eta,0}^{1} (\Xi) \big\Vert\big\Vert \Phi ; \mathring{H}_{\eta,0}^{1} (\Xi) \big\Vert \le C\,|\gamma|\,|\mathfrak{a}_\Xi(V^\eta,V^\eta)|^{1/2}|\mathfrak{a}_\Xi(\Phi,\Phi)|^{1/2}.
\]
Using again (\ref{a0}), this yields the following assertion
\BEP \label{propAP1}
There exist $\eta_0 > 0$ and $\gamma_0 >0$ such that for $\gamma
\in (-\gamma_0; \gamma_0)$, the solution $v^\eta \in \mathring{H}^1(\Xi)$
of Problem \eqref{a7} with the functional $F\in 
\mathring{H}^1(\Xi)^*$ satisfies the estimate
\bea
\big\Vert v^\eta ; \mathring{H}_{\eta,\gamma}^{1}(\Xi)\big\Vert \leq c\, 
\big\Vert F ; \mathring{H}_{\eta,-\gamma}^{1}(\Xi)^* \big\Vert, \label{a17}
\eea
where the coefficient $c$ does not depend on $\eta \geq \eta_0 $ or $\gamma$. 
\ENP
\noindent A similar trick allows one to verify the next claim. 

\BEP  \label{propAP2} 
Assume that Problem \eqref{29} has no trapped mode solutions. 
Then there exists $\gamma_0 > 0$ such that, for $\gamma \in (-\gamma_0;
\gamma_0)$, the integral identity for Problem \eqref{a2} in the space
$\mathring{W}^{\mrm{out}}_\gamma(\Omega_\bullet)$ (see \eqref{a4}),
namely
\bea\label{a18}
\begin{array}{l}
-\dsp\int_{\Omega_\bullet}(\Delta+\lambda)\bigg(\chi\sum_{j=0}^{J-1} a_j w_j^+\bigg)\overline{\varphi}\,dx +(\nabla_x \tilde{u}^\infty, \nabla_x \varphi )_{\Omega_\bullet}
- \lambda ( \tilde{u}^\infty, \varphi )_{\Omega_\bullet}
= F( \varphi)
\end{array}
\eea 
for all $\varphi \in\mathring{W}_{-\beta, -\gamma}^{1}(\Omega_\bullet)$, has a unique solution $u^\infty \in \mathring{W}^{\mrm{out}}_\gamma(\Omega_\bullet)$
with the estimate
\[
\big\Vert u^\infty ; \mathring{W}^{\mrm{out}}_\gamma(\Omega_\bullet) \big\Vert
\leq c\,\big\Vert F ; \mathring{W}_{-\beta, -\gamma}^{1} (\Omega_\bullet)^* \big\Vert.
\]
Here the constant $c$ is independent of $F$ and $\gamma$. Moreover,
when $F \in \mathring{W}_{-\beta, -\gamma}^1 (\Omega_\bullet)^* 
\cap \mathring{W}_{-\beta, 0}^1 (\Omega_\bullet)^* $, this solution belongs
to $\mathring{W}^{\mrm{out}}(\Omega_\bullet)$ and therefore coincides
with the unique solution in this space.  
\ENP
\noindent Let us mention that the result of Proposition \ref{propAP2} can be made stronger by methods of the general theory presented in \cite{Ko1, MaPledge, MaPlpoly} (see also
\cite[Ch.\,8\,\S\,4]{NaPl} and other papers) concerning elliptic boundary value problems in domains with piecewise smooth boundaries. More precisely, introduce the bounded operator $\cA^\infty_{\beta,\gamma}:\mathring{W}^{\mrm{out}}_\gamma(\Omega_\bullet)\to \mathring{W}_{-\beta, -\gamma}^{1}(\Omega_\bullet)^\ast$  such that 
\begin{equation}\label{a20} 
\begin{array}{l}
\langle \cA_\beta^\infty u^{\infty},\varphi\rangle_{\Om_{\bullet}}=-\dsp\int_{\Omega_\bullet}(\Delta+\lambda)\bigg(\chi\sum_{j=0}^{J-1} a_j w_j^+\bigg)\overline{\varphi}\,dx+(\nabla_x \tilde{u}^\infty, \nabla_x \varphi )_{\Omega_\bullet}
- \lambda ( \tilde{u}^\infty, \varphi )_{\Omega_\bullet} 
\end{array}
\end{equation}
for all  $u^{\infty}\in\mathring{W}^{\mrm{out}}_\gamma(\Omega_\bullet),\ \varphi\in\mathring{W}_{-\beta, -\gamma}^{1}(\Omega_\bullet)$. Then regarding the smooth surface $\partial \mathcal{O}$ as a $(d-1)$-dimensional edge and applying the results of the above-mentioned references, one can establish the following assertion. 
\BET \label{theAT4} 
Assume that $\beta$ satisfies (\ref{HypoWeight}) and that
\bea
\gamma \in (-1;1).  \label{a21}
\eea
Then the operator $\cA^\infty_{\beta,\gamma}:\mathring{W}^{\mrm{out}}_\gamma(\Omega_\bullet)\to \mathring{W}_{-\beta, -\gamma}^{1}(\Omega_\bullet)^\ast$   defined in (\ref{a20}) is Fredholm of index zero. 
\ENT
\noindent We emphasize that in contrast to Proposition \ref{propAP2}, Theorem \ref{theAT4}
does not require the absence of trapped modes. It is worth 
mentioning that the requirement $\gamma > -1$  in \eqref{a21} guarantees that
all smooth functions vanishing on $\partial \mathcal{O}$ and having compact support are contained in  $\mathring{W}_{\beta,\gamma}^{1}(\Omega_\bullet)$, while the restriction $\gamma < 1$ excludes from this space the functions with 
singularities of the order $O(\varrho^{-1})$ on $\partial \mathcal{O}$. 
Examples of such functions  are Green's functions having a singular point on
$\partial \mathcal{O}$ and integrated over $\partial \mathcal{O}$ with a smooth density. In the following, Proposition \ref{propAP2} will be sufficient for our needs.\\
\newline
$4^\circ$. Proving the solvability of the original problem for large $\eta$. Introduce the bounded operator $\cB^\eta_{\beta}:W_{\eta}^{\mrm{out}} (\Omega)\to W_{-\beta,\eta}^{1} (\Omega)^*$  (see the definition of theses spaces in (\ref{51}), (\ref{510})) such that 
\[
\langle \cB_\beta^\eta u,v\rangle_{\Om}=-\int_{\Om}(\Delta+\lambda(1+i\eta\,b))\bigg(\chi\sum_{j=0}^{J-1} a_j w_j^+\bigg)\overline{v}\,dx+\int_{\partial\Om}\partial_n\bigg(\chi\sum_{j=0}^{J-1} a_j w_j^+\bigg)\overline{v}\,dx+\mathfrak{a}(\tilde{u},v)
\]
for all $u\in W_{\eta}^{\mrm{out}} (\Omega)$, $v\in W_{-\beta,\eta}^{1} (\Omega)$. Observe that the action of $\cB_\beta^\eta$ is the same as that of the operator $\cA_\beta^\eta : W^{\mrm{out}}(\Omega) \to W_{-\beta}^1 (\Omega)^\ast$ introduced in (\ref{DefCA}). However it is measured with different norms, this is why we give a different name. According to Proposition \ref{propP2} item $2)$, as well (\ref{estima1}), (\ref{a10}), we know that for all $\eta>0$, $\cB_\beta^\eta$ is an isomorphism. Therefore to show Theorem \ref{thT3}, it remains to prove that $\cB_\beta^\eta$ is uniformly invertible for large $\eta$. To proceed, we will construct an ``almost inverse'' operator for $\cB_\beta^\eta$, that is a linear and continuous mapping \bea
\cT_\beta^\eta : W_{-\beta,\eta}^{1} (\Omega)^* \to  
W_{\eta}^{\mrm{out}} (\Omega) \label{a24}  
\eea
such that 
\bea\label{a25}
\big\Vert \cB_\beta^\eta \,\cT_\beta^\eta-{\rm Id} ; W_{-\beta,\eta}^{1} (\Omega)^* \to  W_{-\beta,\eta}^{1} (\Omega)^*\big\Vert 
\leq c \eta^{- \delta}.    
\eea
Here Id is the identity mapping, $\delta$ is some positive number
and $c$ is independent of $\eta \ge \eta_0$ for some $\eta_0>0$. This will guarantee that $\cB_\beta^\eta \,\cT_\beta^\eta$ is an isomorphism for $\eta$ large enough and ensure that $(\cB_\beta^\eta)^{-1}$ is uniformly continuous as we will explain at the end of the proof. \\
\newline
To establish such a result, we apply a trick proposed in \cite{na19}, see also \cite[Ch.\,5]{MaNaPl}, \cite{na239,na564} and others for its usage.
For a functional $F\in W_{-\beta,\eta}^{1} (\Omega)^*$, we set 
\begin{equation}\label{b1} 
\begin{array}{l}
F_\bullet^\eta (\varphi) =  F\big( (1 - \zeta^\eta) \varphi\big),\qquad \forall \varphi  \in \mathring{W}_{-\beta, 0}^{1}(\Omega_\bullet) 
\\[4pt]
F_\circ^\eta (\phi) =  F ( \zeta^\eta \phi ),\qquad \forall \phi  \in \mathring{H}_{\eta,0}^{1} (\Xi),
\end{array}
\end{equation}
where $\zeta^\eta$ is a cut-off function which is equal to 1 in the
$\eta^{-1/4}$-neighborhood of $\overline{\mathcal{O}}$ and vanishes outside
its $2\eta^{-1/4}$-neighborhood.\\
\newline
First, since we have 
\[
{\rm supp} \, F_\bullet^\eta \subset \{ x \in \overline{\Omega} \, |\, 
\varrho(x) \geq \eta^{-1/4} \},
\]
we know that $F_\bullet^\eta\in \mathring{W}_{-\beta, \gamma}^{1}(\Omega_\bullet)^\ast$ for all $\gamma\in\R$. Additionally, for $\delta\ge0$, using that $1\le \eta^\delta\rho^{4\delta}$ as well as $(\varrho + \eta^{-1/2} )^{-1}\le \varrho^{-1}$ on the support of $(1-\zeta^\eta)$, we can write
\[
\begin{array}{rcl}
|F_\bullet^\eta(\varphi) | &\leq & 
c\Vert F ; W_{-\beta,\eta}^{1} (\Omega)^*  \Vert 
\, \Vert (1-\zeta^\eta) \varphi ; W_{-\beta,\eta}^{1} ( \Om)  \Vert\\[4pt]
&\le & C\eta^\delta\Vert F ; W_{-\beta,\eta}^{1} (\Omega)^* \Vert
 \Vert   \varphi ; \mathring{W}_{-\beta,4 \delta}^1 ( \Omega_\bullet)  \Vert.
\end{array}
\]
Thus for all $\delta \geq 0$, there holds
\begin{equation}\label{b4}
\Vert F_\bullet^\eta ; \mathring{W}_{-\beta,4 \delta}^{1} (\Omega_\bullet)^* \Vert \le C \eta^\delta \Vert  F ; W_{-\beta,\eta}^{1} ( \Omega)^*   \Vert .
\end{equation}
Second, in view of the inclusion
\bea\label{b5}
{\rm supp} \, F_\circ^\eta \subset \{ x \in \Xi \, : \, 
{\rm dist} \, (x, \mathcal{O}) \leq 2 \eta^{-1/4} \}, 
\eea
we can write
\[ 
\begin{array}{rcl}
|F_\circ^\eta(\phi) | &\leq & c\Vert F ; W_{-\beta,\eta}^1 (\Omega)^*  \Vert 
\, \Vert \zeta^\eta\phi ; W_{-\beta,\eta}^{1} ( \Om)  \Vert\\[4pt]
&\leq & c\Vert F ; W_{-\beta,\eta}^1 (\Omega)^*  \Vert 
\, \Vert \phi ; \mathring{H}_{\eta,0}^{1} ( \Xi)  \Vert\\[4pt]
 & \le & C\eta^{-\delta}\Vert F ; W_{-\beta,\eta}^1(\Omega)^* \Vert
 \Vert \phi ; \mathring{H}_{\eta, - 4 \delta}^1(\Xi)  \Vert.
\end{array}
\]
Above the small coefficient $\eta^{-\delta}$ appears for the following two reasons: in the definition \eqref{a12} of the norm of the 
space $\mathring{H}_{\eta,\gamma}^1(\Xi)$ with $\gamma = - 4 \delta$, the norms on the subdomain $\mathcal{O}$ are multiplied by the factor $\eta^{4 \delta } \geq c \eta^\delta$
and the norms on the subdomain $\Xi_\bullet$ gain the weight 
$(\varrho(x) + \eta^{-1/2} )^{-4 \delta}$, which exceeds $c \eta^\delta$ for
$x \in \Xi_\bullet \cap {\rm supp}\, F_\circ^\eta$, according to \eqref{b5}. Thus for all $\delta \geq 0$, there holds
\begin{equation}\label{b6}
\Vert F_\circ^\eta ; \mathring{H}_{\eta,-4 \delta}^{1} (\Xi) \Vert
\le  C \eta^{-\delta} \Vert  F ; W_{-\beta,\eta}^1 ( \Omega)^*   \Vert.
\end{equation}
Now assume $\delta \ge 0$ to be small enough so that $4\delta\in[0;\gamma_0)$ where $\gamma_0$ is the smallest of the two weight indices appearing in  Propositions \ref{propAP1} and \ref{propAP2}
(recall that in this section we assume that trapped modes do not exist for Problem \eqref{29}).  This guarantees the existence of solutions
\begin{equation}\label{DefTermCons}
u_\bullet^\eta \in \mathring{W}^{\mrm{out}}_{-4\delta}(\Omega_\bullet)\qquad
\mbox{and }\qquad u_\circ^\eta \in \mathring{H}_{\eta,4\delta}^{1} (\Xi)
\end{equation}
to the problems \eqref{a2} with $F = F_\bullet^\eta$ and \eqref{a7} with $F= F_\circ^\eta$ respectively. Moreover, we get the 
estimates
\begin{equation}\label{b7}
\begin{array}{rcl}
\Vert u_\bullet^\eta ; \mathring{W}^{\mrm{out}}_{-4\delta}(\Omega_\bullet) \Vert 
&\leq & c \Vert F_\bullet^\eta ; \mathring{W}_{-\beta,4 \delta}^{1} (\Omega_\bullet)^* \Vert 
\leq
C\eta^\delta  \Vert F ; W_{-\beta,\eta}^{1} ( \Omega)^*  \Vert,\\[4pt]
\Vert u_\circ^\eta ; \mathring{H}_{\eta,4\delta}^{1} (\Xi) \Vert 
&\leq & c \Vert F_\circ^\eta ; \mathring{H}_{\eta,-4 \delta}^{1} (\Xi)  \Vert 
\leq
C\eta^{- \delta}  \Vert F ;  W_{-\beta,\eta}^1 ( \Omega)^* \Vert .
\end{array}
\end{equation}
Finally we define the function $\cT_\beta^\eta F$ such that
\[
\cT_\beta^\eta F  = \zeta_\bullet^\eta  u_\bullet^\eta + \zeta_\circ  u_\circ^\eta,
\]
where the cut-off functions $\zeta_\bullet^\eta \in \mathscr{C}^\infty_0(\overline{\Omega})$
and $\zeta_\circ \in \mathscr{C}_0^\infty(\Xi)$ are such that
\begin{equation}\label{d0}
\zeta_\bullet^\eta(x)=\begin{array}{|ll}
\,1 & \mbox{for dist$(x, \mathcal{O}) \geq 2 \eta^{-1/2}$}\\[2pt]
\,0 & \mbox{for dist$(x, \mathcal{O}) \leq \eta^{-1/2}$},
\end{array}\quad\qquad \zeta_\circ(x)  = 1 \ \mbox{for dist$(x, \mathcal{O}) \leq d_\circ$ }
\end{equation}
Here we choose $d_\circ >0$ such that the $d_\circ$-neighborhood of $\overline{\mathcal{O}}$ is contained in  $\Xi$. However note that $\zeta_\circ$ is independent of $\eta$. Contrary to the pair $\zeta^\eta$, $1-\zeta^\eta$ introduced in (\ref{b1}), the cut-off functions $\zeta_\bullet^\eta$, $\zeta_\circ$ do not constitute a partition of unity because their supports overlap. This is a crucial ingredient in the method to obtain the desired estimate (\ref{a25}).\\
\newline
Observe first that by using (\ref{b7}) with $\delta=0$, we obtain the equality
\begin{equation}\label{EstimBase}
\Vert \cT_\beta^\eta F ; W_{\eta}^{\mrm{out}} (\Omega) \Vert 
\leq C \Vert F ; W_{-\beta,\eta}^1 (\Omega)^* \Vert
\end{equation}
which will be useful at the end of the proof. Now let us estimate $\big\Vert \cB_\beta^\eta\cT_\beta^\eta-{\rm Id} ; W_{-\beta,\eta}^{1} (\Omega)^* \to  W_{-\beta,\eta}^{1} (\Omega)^*\big\Vert$, that is the left-hand side of \eqref{a25}. A direct calculus yields, for all $v\in \mathscr{C}^\infty_0(\overline{\Om})$,
\begin{equation}\label{d1}
\begin{array}{ll}
 &  \big( \nabla_x( \zeta_\bullet^\eta u_\bullet^\eta + \zeta_\circ u_\circ^\eta ),
\nabla_x v \big)_\Omega
- \lambda \big( (1+  i\eta\,b) (\zeta_\bullet^\eta u_\bullet^\eta + \zeta_\circ u_\circ^\eta ),
v \big)_\Omega - F (v) \\[4pt]
= &\big( \nabla_x u_\bullet^\eta, \nabla_x ( \zeta_\bullet^\eta v ) \big)_{\Omega_\bullet}
- \lambda (  u_\bullet^\eta, \zeta_\bullet ^\eta v )_{\Omega_\bullet} + I_\bullet^\eta(v) \\[3pt]
&\mathfrak{a}_\Xi(u_\circ^\eta, \zeta_\circ v )  +I_\circ^\eta(v) 
 - F (v), 
\end{array}
\end{equation}
with
\begin{equation}\label{d2}
\begin{array}{l}
 I_\bullet^\eta(v) = \big( u_\bullet^\eta \nabla_x \zeta_\bullet^\eta, \nabla_x  v  \big)_{\Omega_\bullet} - 
\big( \nabla_x u_\bullet^\eta, v \nabla_x  \zeta_\bullet^\eta  \big)_{\Omega_\bullet}\\[4pt]
 I_\circ^\eta(v) = \big( u_\circ^\eta \nabla_x \zeta_\circ  
, \nabla_x  v  \big)_\Xi - 
\big( \nabla_x u_\circ^\eta, v \nabla_x  \zeta_\circ  \big)_\Xi.
\end{array}
\end{equation}
Next we work on each of the terms of (\ref{d1}). First, by definition of $u_\bullet^\eta$, $u_\circ^\eta$, we observe that
\begin{equation}\label{d3}
\begin{array}{l}
\big( \nabla_x u_\bullet^\eta, \nabla_x ( \zeta_\bullet^\eta v ) \big)_{\Omega_\bullet} - \lambda (  u_\bullet^\eta, \zeta_\bullet ^\eta v )_{\Omega_\bullet}=F^\eta_\bullet( \zeta_\bullet^\eta v) = F ( (1- \zeta^\eta)\zeta_\bullet^\eta v)= F ( (1- \zeta^\eta) v) \\[4pt]
\mathfrak{a}_\Xi ( u_\circ^\eta, \zeta_\circ    v  )
= F_\circ^\eta( \zeta_\circ v ) = F (\zeta^\eta \zeta_\circ v)= F (\zeta^\eta v).
\end{array}
\end{equation}
Here we used that $(1- \zeta^\eta)\zeta_\bullet^\eta = 1- \zeta^\eta$ and $\zeta^\eta \zeta_\circ= \zeta^\eta$ due to the definition of the cut-off functions. Thus the sum of the terms (\ref{d3}) compensate $F (v)$ in (\ref{d1}). Therefore it remains only to estimate the quantities $I_\bullet^\eta(v)$, $I_\circ^\eta(v)$ defined in \eqref{d2}. \\
\newline 
First, from the definition of $\zeta_\circ$ in \eqref{d0}, we see that supp\,$|\nabla_x \zeta_\circ|
\subset \{ x \in \overline{\Xi} \, : \, {\rm dist}\, (x, \mathcal{O}) \geq d_\circ\}$ so that the integration in $I_\circ^\eta(v)$ is performed at distance from $\mathcal{O}$. Using \eqref{a12}, \eqref{b7}, this allows us to write 
\begin{equation}\label{d5}
\begin{array}{rcl}
|I_\circ^\eta (v)| &\leq  &c\,\big\Vert  v ; \mathring{H}_{\eta,0}^{1} ({\rm supp}\,|\nabla_x \zeta_\circ|) \big\Vert
\, \big\Vert u_\circ^\eta ; \mathring{H}_{\eta,0}^{1} ({\rm supp}\,|\nabla_x \zeta_\circ| )
\big\Vert \\[4pt]
&\le &c\,\big\Vert  v ; W_{-\beta,\eta}^1 (\Omega) \big\Vert
\, \big\Vert u_\circ^\eta ; \mathring{H}_{\eta,4 \delta}^1 ( \Xi) \big\Vert  \\[4pt]
&\le &c\,\eta^{- \delta}\big\Vert  v ; W_{-\beta,\eta}^1 (\Omega) \big\Vert
  \big\Vert F ; W_{-\beta,\eta }^1( \Omega)^* \big\Vert .
\end{array}
\end{equation}
Second, to assess the quantity $I_\bullet^\eta(v)$, we observe that according to \eqref{d0}, we have
\[
{\rm supp}\, |\nabla \zeta_\bullet^\eta| \subset   \big\{ x \in \Xi \,|\,
{\rm dist} \, (x,\mathcal{O}) \in [\eta^{-1/2};2 \eta^{-1/2} ] \big\}.
\]
Therefore on ${\rm supp}\, |\nabla \zeta_\bullet^\eta|$ the weight $1$ which appears in $\mathring{W}^{\mrm{out}}( \Omega_\bullet) $ is bounded by $2^{4\delta}\eta^{-2\delta}$  times $\varrho ^{- 4\delta}$ which is the weight in the norm of $\mathring{W}^{\mrm{out}}_{-4 \delta}( \Omega_\bullet)$ (see \eqref{a4}). This allows us to write, using also \eqref{b7},
\begin{equation}\label{d7}
\begin{array}{rcl}
 |I_\bullet^\eta (v)|  &\leq &
c\,\big\Vert v ; W_{-\beta,\eta}^1 (\Omega) \big\Vert
\, \big\Vert   \tilde{u}_\bullet^\eta ; \mathring{W}^{\mrm{out}}( {\rm supp}\, |\nabla \zeta_\bullet^\eta|) \big\Vert \\[4pt]
&\le & c\,\eta^{-2 \delta}\big\Vert  v ; W_{-\beta,\eta}^1 (\Omega) \big\Vert
  \big\Vert u_\bullet^\eta ; \mathring{W}^{\mrm{out}}_{-4 \delta}( \Omega_\bullet) 
\big\Vert  \\[4pt]
&\le &  c\,\eta^{- \delta} \big\Vert  v ; W_{-\beta,\eta}^1 (\Omega) \big\Vert     \big\Vert F ; W_{-\beta,\eta }^1( \Omega)^* \big\Vert .
\end{array}
\end{equation}
Here we took into account the coincidence of $\tilde{u}^\eta_\bullet$ and $u_\bullet^\eta$ on ${\rm supp}\, |\nabla \zeta_\bullet^\eta|$. Note above that the fact that ${\rm supp}\, |\nabla \zeta_\bullet^\eta|\ne{\rm supp}\, |\nabla \zeta^\eta|$ is essential to get the $\eta^{- \delta}$ at the end of the day. \\
\newline
Gathering \eqref{d5} and \eqref{d7}, finally we obtain \eqref{a25}. It is time now to detail the final step. Denote by $\cR_\beta^\eta:W_{-\beta,\eta}^{1} (\Omega)^* \to  W_{-\beta,\eta}^{1} (\Omega)^*$ the remainder $\cR_\beta^\eta:=\cB_\beta^\eta \,\cT_\beta^\eta-{\rm Id}$. From \eqref{a25}, we know that $\cT_\beta^\eta\,(\mrm{Id} +\cR_\beta^\eta)^{-1}$ is a right inverse for $\cB_\beta^\eta$, i.e. we have $\cB_\beta^\eta\, \cT_\beta^\eta\,(\mrm{Id} +\cR_\beta^\eta)^{-1}=\mrm{Id}$ for $\eta$ large enough. Therefore for a given $F\in W^1_{-\beta}(\Om)^\ast$, the function $u\in W^{\mrm{out}}(\Om)$ satisfying $\mathcal{A}^\eta_\beta u=\cB_\beta^\eta u=F$ is given by $u=\cT_\beta^\eta\,(\mrm{Id} +\cR_\beta^\eta)^{-1}F$. Using (\ref{EstimBase}), finally we can write
\[
\Vert u; W_{\eta}^{\mrm{out}} (\Omega) \Vert 
\leq c \Vert (\mrm{Id} +\cR_\beta^\eta)^{-1}F; W_{-\beta,\eta}^1 (\Omega)^* \Vert \leq C \Vert F ; W_{-\beta,\eta}^1 (\Omega)^* \Vert,
\]
which completes the proof of Theorem \ref{thT3}.\\
\newline
$5^\circ$. Trapped modes. In the statement of the uniform stability estimate of Theorem \ref{thT3}, we assumed that the limit Problem \eqref{29} does not have trapped modes. Let us explain how to proceed when this assumption is not met. To set ideas we assume that the space of trapped modes is of dimension $K$, spanned by some functions $m_1^\tr, \dots, m_K^\tr \in W_\beta^1(\Omega_\bullet) \subset H^1(\Omega_\bullet)$. First, as shown in Section \ref{sec5}, trapped modes do not affect the asymptotics of the scattering matrix because the
diffraction solutions can be subject to the orthogonality conditions 
\eqref{conditionsTrappedModes}. However we need to comment the construction of the almost inverse operator $\cT_\beta^\eta$ in \eqref{a24} which is used in deriving the a priori estimate. Since Problem
\eqref{a2} is no more uniquely solvable, we modify the representation of 
functions belonging to the space $W^{\mrm{out}}_{\eta}(\Omega)$ by setting
\begin{equation}\label{d8}
u = \eta^{1/2} \sum_{k= 1}^K c_k M_k^\tr  +
\chi \sum_{j=0}^{J-1} a_j w_j^+  + \tilde{u},
\end{equation}
where the $M_k^\tr$ are defined as follows:
\begin{equation*}
M_k^\tr (x) = \left\{
\begin{array}{ll}
m_k^\tr (x) + \eta^{-1/2} m_k'(x), & \qquad x\in \Omega_\bullet \\[2pt]
\eta^{-1/2}\chi_\mathcal{O}(x) E(\sqrt{\eta	}n,s) \partial_n m_k^\tr|_{\partial\mathcal{O}} (s), & \qquad x \in \mathcal{O}.
\end{array}
\right.
\end{equation*}
Above $m_k'\in H^1(\Om)$ are some functions supported in $\Xi $ and satisfying the boundary condition
\beas
m_k'(x) = E(0,s) \partial_n m_k^\tr|_{\partial\mathcal{O}} (s), \qquad x \in \partial \mathcal{O},
\eeas
so that there holds $M_k^\tr \in H^1 (\overline{\Omega})$. The
coefficients $c_1, \dots, c_K$ are needed to fulfil the compatibility
conditions \eqref{conditionsTrappedModes} in Problem \eqref{a2} for $u_\bullet^\eta$ (see (\ref{DefTermCons})). The procedure is similar to the one  employed in Section \ref{sec5}. The modification \eqref{d8} requires the following new definition of the
norm in the space $W^{\mrm{out}}_\eta(\Omega)$,
\beas
\big\Vert u ; W^{\mrm{out}}_\eta(\Omega) \big\Vert = \inf \Big\{ \eta^{1/2} 
\sum_{k=1}^K |c_k| + \sum_{j=0}^{J-1} |a_j| + 
\big\Vert \tilde{u} ; W_{\beta,\eta}^{1}(\Omega) \big\Vert \Big\},
\eeas
where the infimum is taken over all  representations \eqref{d8}. Notice
that $M_k^\tr\in W_{\beta,\eta}^1(\Omega)$ and therefore 
representation \eqref{d8} is not unique. This is the very reason for
the transition from the Hilbert norm \eqref{51} to the Banach norm
\eqref{d8}. All other steps in the procedure remain the same.

\end{document}